\documentclass[a4paper,12pt]{article}
\usepackage{amsmath,amssymb,amsthm,textcomp} 
\usepackage{xargs}
\usepackage{enumerate}
\usepackage{multirow}

\usepackage{geometry}
\newcommandx\ModSpace[4][1=S,2=k,3=0,4=N,usedefault]{\mathcal{#1}_{#2}\left(\Gamma_{#3}\left(#4\right)\right)}

\newcommand{\C}{\mathbf{C}}
\newcommand{\condnu}{\mathfrak{c}}
\newcommand{\condeps}{\mathfrak{c}}

\newcommand{\cyclomod}[1][\ell]{\overline{\chi}_{#1}}
\newcommand{\Ebar}{\overline{E}}
\newcommand{\Eprimebar}{\overline{E'}}
\newcommand{\elambda}{\epsilon_{\lambda}}
\newcommand{\Etwobar}{\overline{E}_2}
\newcommand{\Etwoubar}[1]{\overline{E}_{2,#1}}
\newcommand{\Etwou}[1]{E_{2,u}}
\newcommand{\F}{\mathbf{F}}
\newcommand{\Flambda}{\mathbf{F}_{\lambda}}
\newcommand{\fbar}{\overline{f}}
\newcommand{\Fellbar}{\overline{\mathbf{F}}_{\ell}}
\newcommand{\Frob}{\mathrm{Frob}}
\newcommand{\ftilde}{\widetilde{f}}
\newcommand{\Gal}{\mathrm{Gal}}
\newcommand{\Glambdabar}{\overline{G}_{\lambda}}
\newcommand{\GQ}{G_{\Q}}
\newcommand{\GL}{\mathrm{GL}}
\newcommand{\gtilde}{\widetilde{g}}
\newcommand{\ibar}{\underline{i}}
\newcommand{\Id}{\mathrm{Id}}
\newcommand{\lcm}{\mathrm{lcm}}
\newcommand{\localcond}{N(\rhobarss)}
\newcommand{\nubar}{\overline{\nu}}

\newcommand{\PGL}{\mathrm{PGL}}
\newcommand{\projim}{\mathbf{P}(\Glambdabar)}
\newcommand{\projrhobar}{\mathbf{P}(\rhobar_{f,\lambda})}
\newcommand{\Q}{\mathbf{Q}}
\newcommand{\Qbar}{\overline{\mathbf{Q}}}

\newcommand{\rhobar}{\overline{\rho}}
\newcommand{\rhobarss}{\overline{\rho}^{ss}}
\newcommand{\SL}{\mathrm{SL}}

\newcommand{\trivial}{\mathbf{1}}
\newcommand{\Up}[1][p]{\mathcal{U}_{#1}}

\newcommand{\Z}{\mathbf{Z}}

\newtheorem{defi}{Definition}[section]
\newtheorem{lemma}{Lemma}[section]
\newtheorem{prop}{Proposition}[section]
\newtheorem{thm}{Theorem}[section]
\newtheorem*{unnumberedthm}{Theorem}

\title{Explicit Large Image Theorems for Modular Forms}
\author{Nicolas Billerey and Luis V. Dieulefait}

\begin{document}

\maketitle

\begin{abstract}
Let $k$ and $N$ be positive integers with $k\ge2$ even. In this paper we give general explicit upper-bounds in terms of~$k$ and~$N$ from which all the residual representations $\rhobar_{f,\lambda}$ attached to non-CM newforms of weight~$k$ and level~$\Gamma_0(N)$ with $\lambda$ of residue characteristic greater than these bounds are ``as large as possible''. The results split into different cases according to the possible types for the residual images and each of them is illustrated on some numerical examples.
\end{abstract}


    \section*{Introduction}

Let $f$ be a newform of weight $k\ge 2$, level $N\ge 1$ and trivial Nebentypus whose Fourier expansion at infinity is given by $f(\tau)=q+\sum_{n\ge 2}a_nq^n$, with $q=e^{2i\pi\tau}$ and $\tau$ in the complex upper half-plane. We denote by $K$ the number field generated by the coefficients $a_n$ and by $\mathcal{O}$ its ring of integers. Given a prime~$\ell$, we shall denote by $\rho_{f,\ell}$ the $\ell$-adic representation attached to~$f$ by Deligne~:
\begin{equation*}
\rho_{f,\ell}: \Gal(\Qbar/\Q)\longrightarrow \GL(2,\mathcal{O}\otimes_{\Z}\Z_{\ell}).
\end{equation*}
The decomposition $\mathcal{O}\otimes_{\Z}\Z_{\ell}=\prod_{\lambda\mid\ell}\mathcal{O}_{\lambda}$ where the product runs over prime ideals in~$\mathcal{O}$ of residue characteristic~$\ell$, in turn produces for each such~$\lambda$ a representation $\rho_{f,\lambda}$ with values in~$\GL(2,\mathcal{O}_{\lambda})$ where $\mathcal{O}_{\lambda}$ is the completion of~$\mathcal{O}$ at~$\lambda$. Composing it with the reduction map $\GL(2,\mathcal{O}_{\lambda})\rightarrow \GL(2,\Flambda)$, where $\Flambda$ is the residue field of~$\lambda$, finally gives rise to a representation $\rhobar_{f,\lambda}$ which is unique up to semi-simplification.

Let us denote by $\Glambdabar$ the image of $\rhobar_{f,\lambda}$. Using results of Carayol (\cite{Car86}), Ribet proved in~\cite[th.~2.1]{Rib85} the following theorem (for a definition of forms with complex multiplication see~\cite{Rib77} or Def.~\ref{defi:cm_forms}).
\begin{unnumberedthm}[Ribet, 1985]
Assume that $f$ is not a form with complex multiplication. Then for almost all~$\lambda$ (i.e. all but a finite number) the following assertions hold~:
\begin{enumerate}
\item the representation $\rhobar_{f,\lambda}$ is irreducible;
\item the order of the group $\Glambdabar$ is divisible by the residue characteristic of~$\lambda$.
\end{enumerate}
\end{unnumberedthm}
As explained in~\cite[\S3]{Rib85}, this theorem implies that for almost all primes~$\ell$, the image~$G_{\ell}$ of~$\rho_{f,\ell}$ is as ``large'' as possible. Namely, if for simplicity $f$ does not have any inner twist (see~\cite{Rib77} for a definition), the following equality holds for all but finitely many~$\ell$~:
\begin{equation*}
G_{\ell}=\left\{x\in\GL(2,\mathcal{O}\otimes_{\Z}\Z_{\ell})\mid \det(x)\in\Z_{\ell}^{*(k-1)}\right\},
\end{equation*}
where $\Z_{\ell}^{*(k-1)}$ denotes the group of $(k-1)$-th powers in~$\Z_{\ell}^*$.

This theorem is a generalization of~\cite{Rib75} on the case $N=1$, which itself extends pioneer results of Serre (\cite{Ser73}) and Swinnerton-Dyer (\cite{SwD73}) on the the case $N=1$ and $K=\Q$. Although these latter results provide a precise characterization of the prime ideals for which one of the assertions above fails, the general theorem of Ribet is however non-effective. 

The main goal of this paper is to give an effective version of Ribet's theorem, that is a general explicit set of prime numbers depending the weight~$k$ and the level~$N$ such that each representation $\rhobar_{f,\lambda}$ with $f$ newform in~$\ModSpace[S]$ and $\lambda$ of residue characteristic away from this set satisfies the conclusion of Ribet's theorem.

Before describing our main results, we mention that among the special cases covered are a generalization to arbitrary square-free levels of a result of Mazur (\cite{Maz77}) on the so-called Eisenstein primes for weight $2$ and prime level modular forms  and an explicit version of Serre's theorem on the surjectivity of Galois representations attached to elliptic curves over~$\Q$ due to Kraus (\cite{Kra95}) and Cojocaru (\cite{Coj05}). 

Let us denote by $\projrhobar$ the projectivization of~$\rhobar_{f,\lambda}$ and by $\projim$ its image in~$\PGL(2,\Flambda)$. For simplicity, we shall say that $\lambda$ is exceptional if it belongs to the finite set of prime ideals for which one of the assertions of Ribet's theorem does not hold. According to Dickson's classification of subgroups of~$\PGL(2,\Flambda)$ (\cite[Prop.~16]{Ser72}), if $\lambda$ is exceptional (we warn the reader that in the literature, the term ``exceptional'' sometimes refers to the last situation below only), then we have~:
\begin{enumerate}[(i)]
\item\label{item:reducible} either $\rhobar_{f,\lambda}$ is reducible;
\item\label{item:dihedral} or the image $\projim$ in~$\PGL(2,\Flambda)$ is dihedral;
\item\label{item:exceptional} or $\projim$ is isomorphic to $A_4$, $S_4$ or $A_5$.
\end{enumerate}
In each case, we thus provide a divisibility relation or an upper-bound in terms of $k$ and $N$ satisfied by the residue characteristic~$\ell$ of~$\lambda$. A general bound can therefore be obtained by combining the results of the three situations. The last case is the simplest one. Namely we prove~:
\begin{unnumberedthm}[Thm.~\ref{thm:image_isom_A_4}]
If $\projim$ is isomorphic to $A_4$, $S_4$ or $A_5$, then either $\ell\mid N$ or $\ell\le 4k-3$. 
\end{unnumberedthm}
In the second case we give a general upper-bound together with a much finer result in the square-free level case that imply the following~:
\begin{unnumberedthm}[Thm.~\ref{thm:dihedral_case}]
Assume $\projim$ to be dihedral. If $f$ does not have complex multiplication, then we have
\begin{equation*}
\ell\le \left(2\left(4.8kN^2(1+\log\log N)\right)^{\frac{k-1}{2}}\right)^{g_0^{\sharp}(k,N)},
\end{equation*}
where $g_0^{\sharp}(k,N)$ is the number of newforms of weight~$k$ and level~$\Gamma_0(N)$. Besides, if $N$ is square-free, then either $\ell\mid N$, or $\ell\le k$, or $\ell=2k-1$.
\end{unnumberedthm}
The first case is by far the most complicated one and we refer the reader to Theorems~\ref{thm:val_at_2}, \ref{thm:general_thm}, \ref{thm:square_free_level_case} and~\ref{th:last_case} for precise and complete statements. Nevertheless, these results combined with those mentioned in this introduction yield to (slightly stronger versions of) the following theorems in the particular but important cases where $N$ is square-free and $N$ is a square respectively.

\begin{unnumberedthm}[Square-free level case]
Assume that $N=p_1\cdots p_t$ where $p_1,\ldots,p_t$ are $t\ge 1$ distinct primes, is square-free, and $\lambda$ is exceptional. Then, we have~:
\begin{enumerate}
\item either $\ell\in\{p_1,\ldots,p_t\}$;cc
\item or $\ell\le 4k-3$;
\item or $\ell$ divides $\displaystyle{\left\{\begin{array}{ll}
       \gcd_{1\le i\le t}\left(\lcm\left(p_i^{k}-1,p_i^{k-2}-1\right)\right) & \textrm{if }k>2 \\
       \lcm_{1\le i\le t}\left(p_i^{2}-1\right) & \textrm{if }k=2 \\
       \end{array}
\right.}$.
\end{enumerate}
\end{unnumberedthm}
\begin{unnumberedthm}[Square level case]
Assume that $N=c^2$ is a square, $f$ does not have complex multiplication and $\lambda$ is exceptional. Then, we have~:
\begin{enumerate}
\item either $\ell\mid N$
\item or $\ell\le \left(2\left(4.8kN^2(1+\log\log N)\right)^{\frac{k-1}{2}}\right)^{g_0^{\sharp}(k,N)}$, where $g_0^{\sharp}(k,N)$ is the number of newforms of weight~$k$ and level~$\Gamma_0(N)$;
\item or there exists a primitive Dirichlet character~$\nu:(\Z/c\Z)^{\times}\rightarrow\C^{\times}$ such that~:
\begin{enumerate}
\item either $\ell$ divides the norm of $p^k-\epsilon^{-1}(p)$ for some prime $p\mid c$;
\item or $\ell$ divides the numerator of the norm of $B_{k,\epsilon}/2k$
\end{enumerate}
where $c_0$ divides $c$, $\epsilon:(\Z/c_0\Z)^{\times}\rightarrow\C^{\times}$ is the inverse of the primitive Dirichlet character attached to~$\nu^2$ and $B_{k,\epsilon}$ is the $k$-th Bernoulli number attached to~$\epsilon$.
\end{enumerate}

\end{unnumberedthm}

Apart from~(\ref{item:exceptional}) which is slightly different, the main idea in proving the results of the paper is to interpret situations~(\ref{item:reducible}) and~(\ref{item:dihedral}) above in terms of congruences between modular forms. In the case of reducible representations $\rhobar_{f,\lambda}$, the original form $f$ is then shown to be congruent modulo~$\ell$ to a suitable Eisenstein series whose construction depends on the weight and level. The theory of modular forms modulo~$\ell$ of Serre and Katz enables us to interpret this congruence as an equality. The desired bound then follows from a careful study of the constant term of these Eisenstein series at various cusps. Besides, in the case of dihedral projective image, the congruent modular form is a specific twist of the original form~$f$. In that case, the upper-bound follows from those of Sturm and Deligne.

Ghate and Parent recently addressed the question of whether the residual Galois representations attached to rational simple non-CM modular abelian varieties have ``uniform'' large images (see~\cite[Question~1.2]{GhPa12} for a precise statement). A positive answer to their question would follow from the existence of an upper-bound for exceptional primes in the weight $2$ case of Ribet's theorem  depending only on the degree~$[K:\Q]$ (and not on the level~$N$). While we are in contrary working with a fixed level, their work is still quite relevant for us.

The first section of the paper is devoted to classical facts about modular Galois representations and their local behaviors. The next three sections deal with cases~(\ref{item:reducible}), (\ref{item:dihedral}) and~(\ref{item:exceptional}) above respectively. Finally some numerical examples illustrating our results are presented in the last section. 

\medskip

\noindent{\em Acknowledgments. }The first named author is indebted to Mladen Dimitrov, Filippo Nuccio, Nick Ramsey and Panagiotis Tsaknias for helpful conversations. Gabor Wiese deserves special thanks for his constant support and advice as well as for invaluable comments and suggestions. Part of this work was done when N.B. was a postdoc at the Institut f\"ur Experimentelle Mathematik in Essen. He is grateful to his members for a pleasant and stimulative working environment.

    \section{Preliminaries}\label{s:preliminaries}
For simplicity, we shall write $\rho$ and $\rhobar$ for $\rho_{f,\lambda}$ and $\rhobar_{f,\lambda}$ respectively. We further denote by $\rhobarss$ the semi-simplification of~$\rhobar$. In this section, we also assume~$\ell\nmid N$.

    \subsection{Local decomposition at Steinberg primes}\label{ss:Steinberg}
Let $p$ be a prime dividing~$N$ exactly once. We shall write $p\| N$. Under this assumption, the $\ell$-adic repre\-sentation $\rho$ has a unique one-dimensional subspace unramified at~$p$ and the action of a Frobenius at~$p$ on it is given by multiplication by the Fourier coefficient~$a_p$.

We now give a description of $\rhobar_p$ which is defined to be the restriction of $\rhobar$ to a decomposition group $G_p$ at~$p$. For any $x\in \mathcal{O}$, let us denote by $\lambda(x)$ the unramified character of~$G_p$ that maps a Frobenius element to~$x\pmod{\lambda}$. Langlands has proved that (\cite[Prop.~2.8]{LoWe12})
\begin{equation}\label{eq:local_description_at_Steinberg_primes}
\rhobar_p\simeq \begin{pmatrix}
                \mu\cyclomod^{k/2-1} & \star \\
		0 & \mu\cyclomod^{k/2} \\
                \end{pmatrix}
\end{equation}
where $\mu=\lambda(a_p/p^{k/2-1})$ is quadratic since $a_p=\pm p^{k/2-1}$ (\cite[Th.~4.6.17]{Miy06}). In particular, if $\Frob_p\in G_p$ is a Frobenius element at~$p$, then the roots of the characteristic polynomial of $\rhobar(\Frob_p)$ are $a_p\pmod{\lambda}$ and $pa_p\pmod{\lambda}$.

    \subsection{Classification of degeneration cases}\label{ss:degeneration_cases}
Let $N(\rhobarss)$ be the Artin conductor of~$\rhobarss$. It was proved by Carayol that $N(\rhobarss)$ is a divisor of $N$ (\cite{Car86}). Moreover Carayol (\cite{Car89}) and Livn\'e (\cite{Liv89}) have (independently) classified the so-called degeneration cases, that is when $e_p\stackrel{\textrm{def}}{=}v_p(N)-v_p(N(\rhobarss))>0$ for some prime $p$. They proved that when $e_p>0$, we are in one of the situations described in the table below.
\begin{table}[h!]
\centering
\begin{tabular}{|c|c|c|c|}
  \hline
  $v_p(N)$ & $b+1\ge2$ & $1$ & $2$ \\
  \hline
  $v_p(N(\rhobarss))$ & $b\ge 1$ & $0$ & $0$ \\
  \hline
  $e_p$ & $1$ & $1$ & $2$ \\
  \hline
\end{tabular}
\caption{Classification of the degeneration cases}
\label{table:degeneration_cases}
\end{table}

It moreover follows from their classification that in the first and third cases, $p$ satisfies certain congruences modulo~$\ell$. Namely we have the following proposition.
\begin{prop}[Carayol-Livn\'e]\label{prop:Carayol_Livne}
Assume $e_p>0$ and $v_p(N)\ge 2$. Then we have $p\equiv\pm 1\pmod{\ell}$.
\end{prop}

    \subsection{Local description at $\ell$}\label{ss:local_description_at_ell}
Assume $2\le k\le \ell+1$. Let $G_{\ell}$ be a decomposition group at~$\ell$ and $I_{\ell}$ its inertia subgroup. Then Deligne and Fontaine (\cite{Edi92}) have respectively proved that 
\begin{itemize}
\item if $f$ is ordinary at~$\lambda$ (that is if $a_{\ell}\not\equiv0\pmod{\lambda}$), then $\rhobar_{|G_{\ell}}$ is reducible and 
\begin{equation*}
\rhobar_{|I_{\ell}}\simeq\begin{pmatrix}
                         \cyclomod^{k-1} & \star \\
			    0 & 1\\
                         \end{pmatrix};
\end{equation*}
\item if $f$ is not ordinary at~$\lambda$, then $\rhobar_{|G_{\ell}}$ is irreducible and 
\begin{equation*}
\rhobar_{|I_{\ell}}\simeq\begin{pmatrix}
                         \psi^{k-1} & 0 \\
			    0 & \psi'^{k-1}\\
                         \end{pmatrix}
\end{equation*}
where $\{\psi,\psi'\}=\{\psi,\psi^{\ell}\}$ is the set of fundamental characters of level~$2$ (\emph{loc. cit.}, \S2.4).
\end{itemize}
The following lemma is immediate.
\begin{lemma}\label{lemma:cardinality_image}
Assume $\ell>k$. 
\begin{enumerate}
\item The image of $\cyclomod^{k-1}$ is cyclic of order $n=(\ell-1)/\gcd(\ell-1,k-1)\ge 2$. In particular, we have $n=2$ if and only if $\ell=2k-1$. Moreover, if $\ell>4k-3$, then $n>5$.
\item The image of $\psi^{(\ell-1)(k-1)}$ is cyclic of order $m=(\ell+1)/\gcd(\ell+1,k-1)\ge 2$. In particular, we have $m=2$ if and only if $\ell=2k-3$. Moreover, if $\ell>4k-5$, then $m>5$.
\end{enumerate}
\end{lemma}

    \section{Reducible representations}

    \subsection{Preliminaries: Gauss sums and Bernoulli numbers}\label{ss:Gauss_Bernoulli}

Let $\psi:(\Z/f\Z)^{\times}\rightarrow\C^{\times}$ be a primitive Dirichlet character of modulus~$f\ge 1$. The Gauss sum attached to~$\psi$ is defined by
\begin{equation*}
W(\psi)=\sum_{n=1}^{f}\psi(n)e^{2i\pi n/f}.
\end{equation*}
\begin{lemma}\label{lem:Gauss_sum}
We have $\left|W(\psi)\right|=\sqrt{f}$. Moreover, as an algebraic integer, the norm of $W(\psi)$ is a power of~$f$.
\end{lemma}
\begin{proof}
The first part of the lemma is~\cite[Lem.~3.1.1]{Miy06}. Let $\sigma$ be a $\Qbar$-automorphism and $m\in\Z$ such that~$\sigma(e^{2i\pi/f})=e^{2i\pi m/f}$. Then, by~\emph{loc. cit.}, we have~:
\begin{equation*}
\sigma(W(\psi))=\sum_{n=1}^f\psi^{\sigma}(n)e^{2i\pi n m/f}=\overline{\psi^{\sigma}}(m)W(\psi^{\sigma})
\end{equation*}
and thus $\left|\sigma(W(\psi))\right|=\left|W(\psi^{\sigma})\right|=\sqrt{f}$. This completes the proof of the lemma.
\end{proof}
The Bernoulli numbers attached to~$\psi$ are defined by~:
\begin{equation*}
\sum_{n=1}^f\psi(n)\frac{te^{nt}}{e^{ft}-1}=\sum_{m\ge0}B_{m,\psi}\frac{t^m}{m!}.
\end{equation*}
In particular, if $\psi$ is the trivial character, $B_{m,\psi}$ is the classical Bernoulli number~$B_m$, except when $m=1$ in which case $B_{1,\psi}=-B_1=1/2$. The following proposition is a well-known result of van Staudt-Clausen.
\begin{prop}[van Staudt-Clausen]\label{prop:bernoulli_numbers}
Let $m\ge 2$ be an even integer. The denominator of~$B_{m}$ is $\displaystyle{\prod_{p-1\mid m}p}$ where the product runs over the primes~$p$ such that $p-1$ divides~$m$.
\end{prop}
The Bernoulli numbers are also related to certain special values of the $L$-function $L(s,\psi)$ attached to~$\psi$. More precisely, we have the following proposition (\cite[Ch.~4]{Was97}).
\begin{prop}\label{prop:FE_for_L_functions}
Assume $\psi$ to be even. Let $m\ge 2$ be an even integer. Then, we have
\begin{equation*}
L(m,\psi)=-W(\psi)\frac{C_m}{f^m}\cdot\frac{B_{m,\psi^{-1}}}{2m}\not=0,\quad\textrm{where }C_m=\frac{(2i\pi)^m}{(m-1)!}.
\end{equation*}
\end{prop}

    \subsection{Statement of the results}

\begin{thm}\label{thm:val_at_2}
Assume $\rhobar_{f,\lambda}$ to be reducible. If $v_2(N)=2$ or $v_2(N)\ge 3$ is odd, then either $\ell$ divides $N$, or $\ell< k-1$, or $\ell=3$.
\end{thm}
Put $c=\max\{d\ge 1;\ d^2\mid N\}$. The following result is a generalization of Ribet's \cite[Lem.~5.2]{Rib75} on the level~$1$ case to higher levels.
\begin{thm}[main result]\label{thm:general_thm}
Assume $\rhobar_{f,\lambda}$ to be reducible. Then one of the following assertions holds~:
\begin{enumerate}
\item the prime $\ell$ divides $N$ or $\ell< k-1$;
\item the level $N$ is a not square and there exists an even Dirichlet character $\eta:(\Z/c\Z)^{\times}\rightarrow\C^{\times}$ such that for every prime $p$ dividing~$N$ with odd valuation $v_p(N)$, we have
\begin{enumerate}
\item either $v_p(N)\ge3$ and $p\equiv\pm1\pmod{\ell}$;
\item or $v_p(N)=1$ and $\ell$ divides the norm of either $p^{k}-\eta(p)$, or $p^{k-2}-\eta(p)$.
\end{enumerate}
\item the level $N$ is a square (i.e. $N=c^2$) and one of the following holds~:
\begin{enumerate}
\item either there exists a prime $p$ such that $v_p(N)=2$ and $p\equiv \pm1\pmod{\ell}$;
\item or there exists a primitive Dirichlet character~$\nu:(\Z/c\Z)^{\times}\rightarrow\C^{\times}$ such that for $\ell>k+1$ we have~:
\begin{enumerate}
\item either $\ell$ divides the norm of $p^k-\epsilon^{-1}(p)$ for some prime $p\mid c$;
\item or $\ell$ divides the numerator of the norm of $B_{k,\epsilon}/2k$
\end{enumerate}
where $c_0$ divides $c$ and $\epsilon:(\Z/c_0\Z)^{\times}\rightarrow\C^{\times}$ is the inverse of the primitive Dirichlet character attached to~$\nu^2$.
\end{enumerate}
\end{enumerate}
\end{thm}
Note that these two results give an effective bound for~$\ell$ in terms of~$N$ and~$k$ unless $k=2$ and $N=p_1\cdots p_tc^2$ where $p_1,\ldots,p_t$ are $t\ge 1$ distinct primes not dividing~$c$, and $c$ is odd or divisible by~$4$. In the square-free level case (namely when $c=1$), we however have the following theorem whose first part is an immediate corollary of Thm.~\ref{thm:general_thm} and whose second part follows from a generalization of a result of Mazur on the weight $2$ and prime level case (cf.~\cite{Maz77} and \cite[Prop.~1]{MaSe76}).
\begin{thm}[square-free level case]\label{thm:square_free_level_case}
Assume $\rhobar_{f,\lambda}$ reducible and $N=p_1\cdots p_t$ where $p_1,\ldots,p_t$ are $t\ge 1$ distinct primes.
\begin{enumerate}
\item If $k>2$, then one of the following assertions holds~:
\begin{enumerate}
\item either $\ell$ divides $N$ or $\ell<k-1$;
\item $\ell$ divides the following non-zero integer 
\begin{equation*}
\gcd\left(\lcm\left(p_i^k-1,p_i^{k-2}-1\right),1\le i\le t\right).
\end{equation*}
\end{enumerate}
\item If $k=2$ and $\ell\nmid 6N$, then the following assertions hold~:
\begin{enumerate}
\item\label{item:congruence_weight_2} for any $1\le i\le t$ with $a_{p_i}=-1$, we have $p_i\equiv -1\pmod{\ell}$;
\item we have $(a_{p_1},\ldots,a_{p_t})\not=(-1,\ldots,-1)$;
\item if $(a_{p_1},\ldots,a_{p_t})=(+1,\ldots,+1)$, then $\ell$ divides the non-zero integer $\prod_{i=1}^t\left(p_i-1\right)$.
\end{enumerate}
\end{enumerate}
\end{thm}
We point out that Ribet already proved (but did not publish) the second part of this theorem as well as ``converse results'' (see the notes~\cite{Rib10} on his homepage).

The last theorem of this section deals with the cases not covered by the previous results.
\begin{thm}\label{th:last_case}
Assume $\rhobar_{f,\lambda}$ reducible. If $k=2$ and $N$ is of the form $N=p_1\cdots p_tc^2$, where $c\not=1$, $p_1,\ldots,p_t$ are $t\ge 1$ distinct primes not dividing~$c$, and $c$ is odd or divisible by~$4$, then~:
\begin{enumerate}
\item either $\ell\mid N$;
\item or $\ell<k-1$;
\item or there exists a prime~$p$ such that $v_p(N)=2$ and $p\equiv\pm1\pmod{\ell}$;
\item or there exists a primitive Dirichlet character~$\nu:(\Z/c\Z)^{\times}\rightarrow\C^{\times}$ such that for $\ell>3$ we have~:
\begin{enumerate}
\item either $\ell$ divides the norm of $p_i^2-\nu^2(p_i)$ for some $1\le i\le t$;
\item or $\ell$ divides the norm of $p^2-\epsilon^{-1}(p)$ for some prime $p\mid c$;
\item or $\ell$ divides  $p_i-1$ for some $1\le i\le t$;
\item or $\ell$ divides the numerator of the norm of $B_{2,\epsilon}/4$
\end{enumerate}
where $c_0\mid c$ and $\epsilon:(\Z/c_0\Z)^{\times}\rightarrow\C^{\times}$ is the inverse of the primitive Dirichlet character attached to~$\nu^2$.
\end{enumerate}
\end{thm}

    \subsection{The Eisenstein series $E$}\label{ss:E}

Assume $\ell\nmid N$. For simplicity, let us denote $\rhobar$ for $\rhobar_{f,\lambda}$ and assume $\rhobar$ to be reducible. The semi-simplification~$\rhobarss$ of~$\rhobar$ is the direct sum of two characters $\epsilon_1$ and~$\epsilon_2$. Each of them may be decomposed as a product $\nubar_i\cyclomod^{\alpha_i}$ with $\nubar_i$ is unramified at~$\ell$ and $0\le \alpha_i<\ell-1$ ($i=1,2$). Using that $\rhobarss$ has determinant~$\cyclomod^{k-1}$, we get $\alpha_1+\alpha_2\equiv k-1\pmod{\ell-1}$ and $\nubar_2=\nubar_1^{-1}$.

Let us further assume that $\ell+1\ge k$. Using the results of~\S\ref{ss:local_description_at_ell}, one sees that $\{\alpha_1,\alpha_2\}=\{0,k-1\}$ and thus
\begin{equation}\label{eq:Faltings_Jordan_decomposition}
\rhobarss\simeq\nubar\oplus\nubar^{-1}\cyclomod^{k-1},
\end{equation}
with $\nubar\in\{\nubar_1,\nubar_2\}$. Moreover, according to Carayol's theorem of~\S\ref{ss:degeneration_cases}, the conductor $\condnu$ of~$\nubar$ satisfies~:
\begin{equation}\label{eq:divisibility_local_cond_cond}
\localcond=\condnu^2\mid N.
\end{equation}
In particular, $\localcond$ is a square dividing~$N$.

Let $\nu$ be the Teichm\"uller lift of~$\nubar$. We may identify it with a primitive Dirichlet character modulo~$\condnu$. From now on, assume that~:
\begin{enumerate}
\item either $k>2$;
\item or, $k=2$ and $\condnu\not=1$.
\end{enumerate}
Under this assumption, we may consider the Eisenstein series in~$\ModSpace[M][k][0][\condnu^2]$ whose Fourier expansion is given by~:
\begin{equation*}
E(\tau)=-\vartheta(\condnu)\frac{B_k}{2k}+\sum_{n\ge 1}\sigma_{k-1}^{\nu}(n)q^n ,
\end{equation*}
where 
\begin{equation*}
\vartheta(\condnu)=\displaystyle{\left\{\begin{array}{ll}
                                                    1 & \textrm{if }\condnu=1 \\
						    0 & \textrm{otherwise}
                                                    \end{array}
\right.},\quad \sigma_{k-1}^{\nu}(n)=\sum_{0<m|n}\nu(n/m)\nu^{-1}(m)m^{k-1}
\end{equation*}
and $B_k$ is the $k$-th Bernoulli number. Note also that our notation~$E$ differs from the notation~$E_{k}^{\nu,\nu^{-1}}$ of~\cite[Ch.~4]{DiSh05} by a factor~$2$~: $E_{k}^{\nu,\nu^{-1}}=2E$ .

The following proposition gives the constant term of the Fourier expansion of~$E$ at the various cusps of~$\Gamma_0(\condnu^2)$.
\begin{prop}\label{prop:constant_terms_E}
The Eisenstein series $E$ is defined over $\mathcal{O}_L$ where $L$ is the field generated by the values of~$\nu$, unless $\condnu=1$ (and $k>2$) in which case $E$ is the classical Eisenstein series $E_k(\tau)=-B_k/2k+\sum_{n\ge1}\sigma_{k-1}(n)q^n$ of weight~$k$ and level~$1$. Let $s=u/v$ with $\gcd(u,v)=1$, $v\mid \condnu^2$ and $u\pmod{\gcd(v,\condnu^2/v)}$ be a cusp of~$\Gamma_0(\condnu^2)$ and let $\gamma\in\SL(2,\Z)$ such that $\gamma\infty=s$. Then the constant term~$\Upsilon$ of $E|_k\gamma$ is independent of the choice of such a $\gamma$ and satisfies~:
\begin{equation*}
\Upsilon\not=0\Leftrightarrow v=\condnu.
\end{equation*}
In that case, we have~:
\begin{equation*}
\Upsilon=-\nu\left(-u\right)\left(\frac{\condnu}{c_0}\right)^k \frac{W\left((\nu^{2})_0\right)}{W(\nu)}
\frac{B_{k,(\nu^{2})_0^{-1}}}{2k}
\prod_{p\mid \condnu}\left(1-(\nu^2)_0(p)p^{-k}\right),
\end{equation*}
where $(\nu^2)_0$ is the primitive character associated to $\nu^2$ of modulus~$c_0\mid \condnu$. Moreover, if $\condnu>1$, then $E|_k\gamma\in\mathcal{O}_L\left[\frac{1}{\condnu^2}\right](\mu_{\condnu^2})[[q^{1/\condnu^2}]]$ where $\mu_{\condnu^2}$ is the group of $\condnu^2$-th roots of unity.
\end{prop}
\begin{proof}
The proposition is immediate when $\condnu=1$. Assume therefore $\condnu>1$. Then by construction the Fourier expansion of~$E$ has coefficients in~$\mathcal{O}_L$ and therefore $E$ is defined over~$\mathcal{O}_L\left[1/\condnu^2\right](\mu_{\condnu^2})$ (\cite[\S1.6]{Kat73}).

Let $s=u/v$ as in the proposition be a cusp of~$\Gamma_0(\condnu^2)$ (for the description of a set of representatives of the cusps of~$\Gamma_0(\condnu^2)$, see~\cite[Prop.~2.6]{Iwa97}) and $\gamma\in\SL(2,\Z)$ such that $\gamma\infty=s$. The last assertion follows from the $q$-expansion principle and the fact that the Fourier of~$E$ at~$\infty$ has coefficients in~$\mathcal{O}_L$ (\cite[Cor.~1.6.2.]{Kat73}).

Since $k$ is even, the constant term of $E$ at~$s$ is well-defined (i.e. does not depend of the choice of such a~$\gamma$). Put
\begin{equation}\label{eq:def_of_G}
\gamma=\displaystyle{\begin{pmatrix}
                                                                           u & \beta \\
									   v & \delta \\
                                                                           \end{pmatrix}}\in\SL(2,\Z)\quad\textrm{and}\quad 
G=\frac{C_kW(\nu)}{\condnu^k}E,\quad\textrm{where }C_k=\frac{(2i\pi)^k}{(k-1)!}.
\end{equation}
The constant part of $G|_k\gamma$ is then given by the following sum (see~\cite[Ch.~4]{DiSh05} and~\cite[\S~VII.3]{Sch74} for a justification in the weight~$2$ case; the factor $1/2$ comes from our normalization for~$E$)~:
\begin{equation*}
\Upsilon_0=\frac{1}{2}\sum_{i,j,l=0}^{\condnu-1}\nu(ij)\vartheta\left(\overline{i\condnu u+v(j+l\condnu)}\right)\zeta^{\overline{\condnu i\beta+(j+l\condnu)\delta}}(k),
\end{equation*}
where the bar means reduction modulo~$\condnu^2$,
\begin{equation*}
\vartheta(\overline{n})=\displaystyle{\left\{\begin{array}{ll}
                                                    1 & \textrm{if }n\equiv0\pmod{\condnu^2} \\
						    0 & \textrm{otherwise}
                                                    \end{array}
\right.},\quad \zeta^{\overline{n}}(k)=\displaystyle{\sideset{}{'}\sum\limits_{m\equiv n\pmod{\condnu^2}}\frac{1}{m^k}},
\end{equation*}
and the primed summation notation means to sum over non-zero integers. 

Assume $\Upsilon_0$ to be non-zero. Then, there exist $i,j,l\in\{0,\ldots,\condnu-1\}$ such that 
\begin{equation*}
\nu(ij)\vartheta\left(\overline{i\condnu u+v(j+l\condnu)}\right)\not=0.
\end{equation*}
In other words, $\gcd(ij,\condnu)=1$ and $i\condnu u+v(j+l\condnu)\equiv 0\pmod{\condnu^2}$. It follows that  $vj\equiv 0\pmod{\condnu}$. But $j$ is co-prime to~$\condnu$ by assumption. So, $v\equiv 0\pmod{\condnu}$ and $u$ is invertible modulo~$\condnu$. The congruence $i\equiv -(j/u)(v/\condnu)\pmod{\condnu}$ follows easily and therefore, we have~:
\begin{equation*}
\nu(ij)=\nu\left(-\frac{vj^2}{u\condnu}\right)=\nu\left(-\frac{j^2}{u}\right)\nu\left(\frac{v}{\condnu}\right)\not=0.
\end{equation*}
So, $\gcd\left(v/\condnu,\condnu\right)=1$ and since $\condnu\mid v$ and $v\mid\condnu^2$, we get $v=\condnu$.

Conversely, assume $v=\condnu>1$. Then, $\gcd(u,\condnu)=1$ and on one hand, we have~:
\begin{equation*}
i\condnu u+v(j+l\condnu)\equiv 0\pmod{\condnu^2}\Longleftrightarrow i\equiv -j/u\pmod{\condnu}
\end{equation*}
and on the other hand~:
\begin{align*}
\condnu i\beta+(j+l\condnu)\delta & =  \frac{1}{u}\left(u\condnu i\beta+(j+l\condnu)u\delta\right)\\
		     & \equiv   \frac{1}{u}\left(-vj\beta+(j+l\condnu)(1+\beta v\right)\pmod{\condnu^2}\\
		     & \equiv   \frac{1}{u}\left(j+l\condnu\right)\pmod{\condnu^2}.
\end{align*}
Combining these two facts, we find that~:
\begin{align}\label{eq:constant_term}
2\Upsilon_0 & =  \sum_{l=0}^{\condnu-1}\sum_{\substack{j=0 \\ \gcd(j,\condnu)=1}}^{\condnu-1}{\nu(-j^2/u)}\zeta^{\overline{\frac{j+l\condnu}{u}}}(k) \nonumber \\
	  & =  \nu(-u)\sum_{l=0}^{c-1}\sum_{\substack{j=0 \\ \gcd(j,\condnu)=1}}^{\condnu-1}{\nu\left(j^2/u^2\right)}
	  \sideset{}{'}\sum_{m\equiv(j+l\condnu)/u\pmod{\condnu^2}}\frac{1}{m^k} \\
	  & = \nu(-u)\sum_{\substack{j=0 \\ \gcd(j,\condnu)=1}}^{\condnu-1}
	  \sideset{}{'}\sum_{m\equiv j/u\pmod{\condnu}}\frac{\nu^2(m)}{m^k} \nonumber \\
	  & = 2\nu(-u)\sum_{m\ge1}\frac{\nu^2(m)}{m^k}= 2\nu(-u)L(k,\nu^2), \nonumber
\end{align}
where $\nu^2$ is viewed as a character modulo~$\condnu$. Let $(\nu^2)_0$ be the primitive Dirichlet character attached to~$\nu^2$. It is an even character modulo~$c_0\mid \condnu$ and we have~:
\begin{equation}
L(k,\nu^2)=L\left(k,(\nu^2)_0\right)\prod_{p\mid \condnu}\left(1-(\nu^2)_0(p)p^{-k}\right).
\end{equation}
Applying Prop.~\ref{prop:FE_for_L_functions} to $\psi=(\nu^2)_0$ and $m=k$, we get~:
\begin{equation}\label{eq:FE}
L(k,(\nu^2)_0)=-W\left((\nu^{2})_0\right)\frac{C_k}{c_0^k}\frac{B_{k,(\nu^{2})_0^{-1}}}{2k}\not=0.
\end{equation}
According to Eq.~(\ref{eq:constant_term})-(\ref{eq:FE}) together with~(\ref{eq:def_of_G}), when $v=\condnu$, the constant term of the Fourier expansion of $E$ at~$s$ is thus the non-zero algebraic number~:
\begin{equation*}
\Upsilon=\frac{\condnu^k}{C_kW(\nu)}\Upsilon_0=-\nu(-u)\left(\frac{\condnu}{c_0}\right)^k \frac{W\left((\nu^{2})_0\right)}{W(\nu)}
\frac{B_{k,(\nu^{2})_0^{-1}}}{2k}
\prod_{p\mid \condnu}\left(1-(\nu^2)_0(p)p^{-k}\right),
\end{equation*}
as claimed.
\end{proof}

    \subsection{Proof of Theorems~\ref{thm:val_at_2} and~\ref{thm:general_thm}}\label{ss:proof_general_thm}
Assume~$\rhobar$ reducible with $\ell\nmid N$ and $\ell+1\ge k$. We keep the notation of~\S\ref{ss:E}. In particular, we have (cf.~(\ref{eq:Faltings_Jordan_decomposition}) and~(\ref{eq:divisibility_local_cond_cond}))
\begin{equation}\label{eq:Faltings_Jordan_decomposition_general_thm}
\rhobarss\simeq\nubar\oplus\nubar^{-1}\cyclomod^{k-1},
\end{equation}
where $\nubar$ is a character of conductor~$\condnu$ such that $\condnu^2\mid N$. So, in particular, we have~$\condnu\mid c$.

Assume that $v_2(N)=2$. Then, $v_2(c)=1$ and $\condnu$ is odd since there is no primitive Dirichlet character modulo twice an odd integer. Therefore, we are in a degeneracy case at~$p=2$ as described in~\S\ref{ss:degeneration_cases}. By Prop.~\ref{prop:Carayol_Livne}, we have $2\equiv \pm1\pmod{\ell}$, namely $\ell=3$.

If $N$ is not a square, let us consider a prime $p$ dividing~$N$ with odd valuation~$v_p(N)$. Once again, we necessarily are in one of the degeneration cases. If $v_p(N)\ge 3$, then by Prop.~\ref{prop:Carayol_Livne}, we get $p\equiv\pm1\pmod{\ell}$. This completes the proof of Thm.~\ref{thm:val_at_2}.

Assume now that for some prime~$p$, we have $v_p(N)=1$ and let us denote by $\eta$ the Teichm\"uller lift of $\nubar^2$. Since $\condnu$ is a divisor of~$c$, we may identify $\eta$ with an even Dirichlet character modulo~$c$. Comparing the restriction to a decomposition group at~$p$ of~$\rhobarss$ given by~(\ref{eq:Faltings_Jordan_decomposition}) with the local representation given by~(\ref{eq:local_description_at_Steinberg_primes}) we get the following equality between sets of characters of~$G_p$~:
\begin{equation*}
\left\{\nubar,\nubar^{-1}\cyclomod^{k-1}\right\}=\left\{\mu\cyclomod^{k/2},\mu\cyclomod^{k/2-1}\right\},
\end{equation*}
where $\mu=\lambda(a_p/p^{k/2-1})$ is the quadratic character defined in~\S\ref{ss:Steinberg}.
We thus are in one of the following situations~:
\begin{enumerate}
\item Either $\nubar=\mu\cyclomod^{k/2}$ and then $\nubar^2=\cyclomod^k$. Applying this equality to a Frobenius element at~$p$, we get that $\nubar^2(\Frob_p)=p^k\pmod{\ell}$ and therefore $\ell$ divides the norm of $p^k-\eta(p)$.
\item Or $\nubar=\mu\cyclomod^{k/2-1}$ and then $\nubar^2=\cyclomod^{k-2}$. Again we have $\nubar^2(\Frob_p)=p^{k-2}\pmod{\ell}$ and we conclude as before that $\ell$ divides the norm of $p^{k-2}-\eta(p)$.
\end{enumerate}

It remains to prove Thm.~\ref{thm:general_thm} when $N$ is a square, namely when $N=c^2$. Assume first that $\condnu\not=c$. Then we are in a degeneracy case as described in~\S\ref{ss:degeneration_cases} for some prime number~$p$. Moreover, $\localcond=\condnu^2$ is a square and therefore we have $v_p(N)=2$ and $v_p(\localcond)=0$. By Prop.~\ref{prop:Carayol_Livne}, it follows that $p\equiv\pm1\pmod{\ell}$. 

In other words, if for every prime $p$ dividing~$N$ with valuation~$2$, we have $p\not\equiv\pm1\pmod{\ell}$, then $\condnu=c$, $N=\condnu^2$ and there is no degeneration at all. Assume now that we are in this situation. Since the space of weight~$2$ and level~$1$ modular forms is trivial, it follows that either $k>2$, or $k=2$ and $\condnu\not=1$. Therefore we may consider the Eisenstein series $E$ of~\S\ref{ss:E}. Let $M$ denote the compositum of~$K$ and~$L$ (the field generated by the values of~$\nu$).

\begin{lemma}\label{lem:strong_cong_f_E}
The Eisenstein series $E$ is a normalized eigenform for all the Hecke operators at level~$\Gamma_0(N)$. Moreover, there exists a prime ideal~$\mathcal{L}$ above~$\ell$ in the integer ring of~$M$ such that~:
\begin{equation*}
a_r\equiv a_r(E)\pmod{\mathcal{L}},\quad\textrm{for all primes }r\not=\ell.
\end{equation*}
\end{lemma}
\begin{proof}
The fact that $E$ is a normalized eigenform for all the Hecke ope\-rators at level~$\Gamma_0(N)$ follows for instance from~\cite[Prop.~5.2.3]{DiSh05}. Moreover by isomorphism~(\ref{eq:Faltings_Jordan_decomposition_general_thm}) there exists a prime ideal~$\mathcal{L}$ above~$\ell$ in the integer ring of~$M$ such that~:
\begin{equation*}
a_r\equiv a_r(E)\pmod{\mathcal{L}},\quad\textrm{for all primes }r\nmid\ell N.
\end{equation*}
If now $r$ is a prime dividing~$N$, then $r^2\mid N$ and $a_r=0$ (\cite[Th.~4.6.17]{Miy06}). Besides, $\nu(r)+\nu^{-1}(r)r^{k-1}=0$. Hence $a_r=0=a_r(E)$. This proves the lemma.
\end{proof}

Let now $\Theta$ be the Katz' operator on modular forms over~$\Fellbar$ whose action on $q$-expansions is given by $q\frac{d}{dq}$ (denoted $A\theta$ in~\cite{Kat77}). Assume $\ell>k+1$. Then the constant term of~$E$ at~$\infty$ is non-zero only if $\condnu=1$ and $k>2$. In that case it is $-B_k/2k$ which is $\ell$-integral by Prop.~\ref{prop:bernoulli_numbers}. We denote by $\fbar$ and $\Ebar$ the modular forms over~$\Fellbar$ obtained by reduction modulo~$\mathcal{L}$ of $f$ and $E$ respectively. Lemma~\ref{lem:strong_cong_f_E} implies that $\Theta(\fbar)=\Theta(\Ebar)$. Moreover Katz has proved that if $\ell>k+1$, then $\Theta$ is injective (\cite[Cor.~(3)]{Kat77}). Under this assumption, it thus follows that the Eisenstein series $E$ becomes cuspidal after reduction. 

If $\condnu=1$ we immediately get that $\ell$ divides the numerator of~$B_k/2k$ as stated in the theorem. Assume therefore that $\condnu>1$. Then $\ell$ divides the numerator of the norm of the constant term of~$E$ at each cusp of~$\Gamma_0(\condnu^2)$, namely by Prop.~\ref{prop:constant_terms_E}~:
\begin{equation*}
\Upsilon=\pm\left(\frac{\condnu}{c_0}\right)^k \frac{W\left(\epsilon^{-1}\right)}{W(\nu)}
\frac{B_{k,\epsilon}}{2k}
\prod_{p\mid \condnu}\left(1-\epsilon^{-1}(p)p^{-k}\right),
\end{equation*}
where $\epsilon:(\Z/c_0\Z)^{\times}\rightarrow\C^{\times}$ is the inverse of the primitive Dirichlet character attached to~$\nu^2$.
By Lemma~\ref{lem:Gauss_sum}, the prime divisors of the norm of~$W\left(\epsilon^{-1}\right)/W(\nu)$ divide~$N$ and therefore are co-prime to~$\ell$. The same obviously holds for $\condnu/c_0$. Therefore we eventually get that $\ell$ divides the norm of either~$p^k-\epsilon^{-1}(p)$ for some~$p$ dividing $\condnu$ (and thus $c$) or the norm of the numerator of~$B_{k,\epsilon}/2k$. This completes the proof of Thm.~\ref{thm:general_thm}.

    \subsection{Proof of Theorem~\ref{thm:square_free_level_case}}

As already mentioned, the first part of Thm.~\ref{thm:square_free_level_case} is a direct corollary of Thm.~\ref{thm:general_thm}. So, let us assume $k=2$ and $\ell\nmid 6N$. By the reasoning at the beginning of~\S\ref{ss:E}, we may write~:
\begin{equation}\label{eq:Faltings_Jordan_weight_2}
\rhobarss\simeq\trivial\oplus\cyclomod,
\end{equation}
where $\trivial$ is the trivial character of~$\Gal(\Qbar/\Q)$. In particular, we have $\condnu^2=\localcond=1$, hence $\condnu=1$. Let now $p\in\{p_1,\ldots,p_t\}$ be a prime dividing~$N$. By~\S\ref{ss:Steinberg}, the local representation~$\rhobar_p$ at~$p$ semi-simplifies to~:
\begin{equation}\label{eq:Steinberg_weigh_2}
\lambda(a_p)\oplus\lambda(a_p)\cyclomod.
\end{equation}
Comparing~(\ref{eq:Faltings_Jordan_weight_2}) and~(\ref{eq:Steinberg_weigh_2}) we get the following  equality between sets of characters of~$G_p$~:
\begin{equation*}
\left\{\trivial,\cyclomod\right\}=\left\{\lambda(a_p)\cyclomod,\lambda(a_p)\right\}.
\end{equation*}
If moreover $a_p=-1$, then the character $\lambda(a_p)$ is non-trivial and therefore, we must have $\lambda(a_p)=\cyclomod$ as characters of~$G_p$. In other words, $p\equiv-1\pmod{\ell}$. This proves assertion~(\ref{item:congruence_weight_2}) of Thm.~\ref{thm:square_free_level_case}.

Before proving the next two assertions, note that we precisely are in the excluded situation of~\S\ref{ss:E}, namely $k=2$ and $\condnu=1$. For that reason, we cannot use the Eisenstein series~$E$ as in the proof of Thm.~\ref{thm:general_thm} (cf.~\S\ref{ss:proof_general_thm}).

To circumvent the lack of weight $2$ level $1$ Eisenstein series, it will be more convenient to directly work with modular forms over~$\Fellbar$. Let $\Etwobar$ be the reduction modulo~$\ell$ (recall that $\ell\ge 5$) of the classical series $E_2$ in charac\-teristic~$0$ defined by~:
\begin{equation*}
E_2(\tau)=-\frac{1}{24}+\sum_{n\ge 1}\sigma_1(n)q^n .
\end{equation*}
Viewed as a modular form over~$\Fellbar$ of level~$N$ (which is co-prime to~$\ell$ by assumption), $\Etwobar$ has filtration~$\ell+1$ (\cite{Ser73}). Put~:
\begin{equation*}
E'=\left[\prod_{p\mid N}\left(a_p\Up -p\Id\right)\right]\Etwobar.
\end{equation*}
The following proposition summarizes the main properties of~$E'$.
\begin{prop}
As a modular form over~$\Fellbar$, $E'$ is a well-defined normalized eigenform for all the Hecke operators at level~$\Gamma_0(N)$ such that~:
\begin{equation*}
\left\{
\begin{array}{rcll}
T_rE' & = & (1+r)E' & \textrm{for all prime }r\nmid N\\
\Up E' & = & a_pE' & \textrm{for any prime }p\mid N.
\end{array}
\right.
\end{equation*}
Moreover $E'$ has filtration~$2$ unless $(a_{p_1}(f),\ldots,a_{p_t}(f))=(-1,\ldots,-1)$ when it has filtration~$\ell+1$. The constant term of its Fourier expansion at infinity is given by~:
\begin{equation*}
a_0(E')=\left\{
\begin{array}{cl}
(-1)^{t+1}\frac{(p_1-1)\cdots(p_t-1)}{24} & \textrm{if }(a_{p_1}(f),\ldots,a_{p_t}(f))=(+1,\ldots,+1) \\
0 & \textrm{otherwise}.
\end{array}
\right.
\end{equation*}
 
\end{prop}
\begin{proof}
By the commutativity of the Hecke algebra, $E'$ is a well-defined modular form over~$\Fellbar$ of level~$N$. Let $r$ be a prime not dividing~$N$. Since $T_r\Etwobar=(1+r)\Etwobar$, we get that $T_rE'=(1+r)E'$, as claimed.

Let $u\not=1$ be an integer dividing~$N$. We denote by $\Etwoubar{u}$ the reduction modulo~$\ell$ of the classical characteristic-$0$ Eisenstein series $\Etwou{u}\in\ModSpace[M][2][0][u]$ defined by~:
\begin{equation}\label{eq:def_E2u}
\Etwou{u}(\tau)=E_2(\tau)-uE_2(u\tau)=\frac{u-1}{24}+\sum_{n\ge 1}\Big(\sum_{\substack{0<m\mid n\\ u\nmid m}}m\Big)q^n .
\end{equation}
If $p$ is a prime divisor of~$N$, recall that we have~:
\begin{equation*}
\begin{array}{rcl}
\Up\Etwobar & = & \Etwoubar{p}+p\Etwobar;\\
\Up\Etwoubar{u} & = &
\left\{
\begin{array}{ll}
\Etwoubar{p}+(1+p)\Etwoubar{u}-\Etwoubar{pu} & \textrm{if }p\nmid u \\
\Etwoubar{p}+p\Etwoubar{u/p} & \textrm{if $p\mid u$ and }p\not=u  \\
\Etwoubar{p} & \textrm{if }p=u.
\end{array}
\right.
\end{array}
\end{equation*}
So, let $p$ be a prime divisor of~$N$. We have~:
\begin{align*}
(a_p\Up-p\Id)\Up\Etwobar & = ((a_p\Up-p\Id))(\Etwoubar{p}+p\Etwobar) \\
			    & = p^2(a_p-1)\Etwobar+(a_p-p+pa_p)\Etwoubar{p}.
\end{align*}
If $a_p=+1$, then we get $(a_p\Up-p\Id)\Up\Etwobar=\Etwoubar{p}=(a_p\Up-p\Id)\Etwobar$ which is the desired result. On the other hand, if $a_p=-1$, then, by the assertion~(\ref{item:congruence_weight_2}) proved above, we have $p\equiv -1\pmod{\ell}$ and the previous equality between forms over~$\Fellbar$ thus gives~:
\begin{equation*}
(a_p\Up-p\Id)\Up\Etwobar=-2\Etwobar+\Etwoubar{p}=-(a_p\Up-p\Id)\Etwobar.
\end{equation*}
To finish the proof, it now remains to compute the filtration of~$E'$ and the first two terms of its Fourier expansion at infinity. Let $s=\sharp\{1\le i\le t\mid a_{p_i}(f)=+1\}$. If $0<s<t$, we may assume without loss of generality that~:
\begin{equation*}
N=p_1\cdots p_s\cdot p_{s+1}\cdots p_t\quad\textrm{with }
\left\{
\begin{array}{rcll}
\Up[p_i] f & = & f & \textrm{for all }1\le i\le s\\
\Up[p_i] f & = & -f & \textrm{for all }s+1\le i\le t.
\end{array}
\right.
\end{equation*}
By induction on~$t$, we prove that~:
\begin{equation*}
E'=\delta_{(s=0)}2^t\Etwobar+
\sum_{\substack{(k,l)\in \{0,\ldots,s\}\times\{0,\ldots,t-s\} \\ (k,l)\not=(0,0)}}(-1)^{k+1}\sum_{\substack{1\le i_1<\cdots<i_k\le s \\ s+1\le j_1<\cdots<j_l\le t}}\Etwoubar{p_{i_1}\cdots p_{i_k}\cdot p_{j_1}\cdots p_{j_l}}
\end{equation*}
where
\begin{equation*}
\delta_{(s=0)}=\displaystyle{\left\{\begin{array}{ll}
                                                    1 & \textrm{if }s=0 \\
						    0 & \textrm{otherwise}
                                                    \end{array}
\right.}
\end{equation*}
and the condition $1\le i_1<\cdots<i_k\le s$ or $s+1\le j_1<\cdots<j_l\le t$ is empty if $s=0$ or $s=t$ respectively. From this equality it follows the assertion on the filtration. Moreover an easy computation using Newton's binomial theorem and~(\ref{eq:def_E2u}) proves the assertions on the first two Fourier coefficients.
\end{proof}
Let us now finish the proof of Thm.~\ref{thm:square_free_level_case}. According to~(\ref{eq:Faltings_Jordan_weight_2}) and the previous proposition, we have~:
\begin{equation*}
a_n(\fbar)=a_n(E')\quad\textrm{for all prime-to-$\ell$ integers~$n$},
\end{equation*}
where $\fbar$ denotes the modular form over~$\Fellbar$ obtained by reduction of~$f$ modulo~$\lambda$. Since $\ell\ge 5>k+1=3$, Katz' theory (\cite[Cor.~(3)]{Kat77}) actually shows that $\fbar=E'$. Thus $E'$ has filtration~$2$ and we cannot have $(a_{p_1}(f),\ldots,a_{p_t}(f))=(-1,\ldots,-1)$. Moreover, the constant term of~$E'$ at infinity must vanish and when $(a_{p_1}(f),\ldots,a_{p_t}(f))=(+1,\ldots,+1)$, this gives the congruence stated in the theorem.

    \subsection{Proof of Theorem~\ref{th:last_case}}

Assume~$\rhobar$ reducible with $\ell\nmid N$ and $\ell+1\ge k$. As in \S~\ref{ss:proof_general_thm}, we have 
\begin{equation}\label{eq::Faltings_Jordan_decomposition_last_case}
\rhobarss\simeq\nubar\oplus\nubar^{-1}\cyclomod
\end{equation}
where $\nubar$ is a character of conductor~$\condnu$ such that $N(\rhobarss)=\condnu^2\mid N$. So, in particular, we have~$\condnu\mid c$. 

If $\condnu\not=c$, then we necessarily are in a degeneracy case as described in~\S\ref{ss:degeneration_cases}, with~$e_p=2$ at some prime divisor~$p$ of~$c$. Therefore, $v_p(N)=2$ and by Prop.~\ref{prop:Carayol_Livne}, we have $p\equiv\pm1\pmod{\ell}$.

We can thus assume, from now on, that $\condnu=c$. Let us denote by $\nu$ the Teichm\"uller lift of~$\nubar$, viewed as a primitive Dirichlet character modulo~$c$. 

Let $1\le i\le t$. Comparing the restriction to a decomposition group at~$p_i$ of~$\rhobarss$ with the local representation given by~(\ref{eq:local_description_at_Steinberg_primes}) we get the following equality between sets of characters of~$G_{p_i}$~:
\begin{equation*}
\left\{\nubar,\nubar^{-1}\cyclomod\right\}=\left\{\lambda(a_{p_i})\cyclomod,\lambda(a_{p_i})\right\},
\end{equation*}
where $\lambda(a_{p_i})$ is the quadratic character defined in~\S\ref{ss:Steinberg}.

Assume that for some $1\le i\le t$, we have $\nubar=\lambda(a_{p_i})\cyclomod$ (again, as characters of~$G_{p_i}$). Since $a_{p_i}=\pm1$, it then follows that $\ell$ divides the norm of $\nu(p_i)^2-p_i^2$. 

From now on, we will therefore assume that $\nubar=\lambda(a_{p_i})$ for every $1\le i\le t$. It then follows that $\nubar(p_i)=a_{p_i}\pmod{\ell}$. Since $\condnu>1$, we may consider the Eisenstein series
\begin{equation*}
E(\tau)=\sum_{n\ge 1}\sigma_{1}^{\nu}(n)q^n \in\ModSpace[M][2][0][\condnu^2]
\end{equation*}
introduced in \S\ref{ss:E}. This is an eigenform for all the Hecke operators at level~$\Gamma_0(\condnu^2)$.

  \subsubsection{The Eisenstein series $E'$}

Put 
\begin{equation*}
E'(\tau)=\left[\prod_{i=1}^t(\mathcal{U}_{p_i}-p_i\nu^{-1}(p_i)\Id)\right]E(p_1\cdots p_t\tau)\in\ModSpace[M][2],
\end{equation*}
where $\mathcal{U}_{p_i}$ denotes the $p_i$-th Hecke operator acting on~$\ModSpace[M][2]$.
In expanded form, we have~:
\begin{equation}\label{eq:E'_expanded_form}
E'(\tau)=E+\sum_{j=1}^t(-1)^j\sum_{1\le i_1<\cdots<i_j\le t}p_{i_1}\cdots p_{i_j}\nu^{-1}(p_{i_1}\cdots p_{i_j}) E(p_{i_1}\cdots p_{i_j}\tau).
\end{equation}
As before let us denote by $L$ the field generated by the values of~$\nu$ and by $M$ the compositum of $L$ and~$K$. The following lemma is crucial.
\begin{lemma}\label{lem:strong_cong_f_E'}
The Eisenstein series $E'$ is a normalized eigenform for all the Hecke operators at level~$\Gamma_0(N)$. Moreover, there exists a prime ideal~$\mathcal{L}$ above~$\ell$ in the integer ring of~$M$ such that~:
\begin{equation*}
a_r\equiv a_r(E')\pmod{\mathcal{L}},\quad\textrm{for all primes r}\not=\ell.
\end{equation*}
\end{lemma}
\begin{proof}
The Eisenstein series $E'$ is clearly normalized and since $\ell$ is co-prime to~$N$, this is an eigenfunction for the $T_{\ell}$-operator acting on~$\ModSpace[M][2]$. By isomorphism~(\ref{eq::Faltings_Jordan_decomposition_last_case}) and assumption $\nubar(p_i)=a_{p_i}\pmod{\ell}$, $1\le i\le t$, there exists a prime ideal~$\mathcal{L}$ above~$\ell$ in the integer ring of~$M$ such that~:
\begin{equation*}
\nu(r)+\nu^{-1}(r)r\equiv a_r\pmod{\mathcal{L}},\quad\textrm{ for every prime }r\nmid\ell N
\end{equation*}
and $\nu(p_i)\equiv a_{p_i}\pmod{\mathcal{L}}$ for any $1\le i\le t$. Let $r$ be a prime. If $r$ does not divide $\ell N$, then $E'$ is a $T_r$-eigenfunction with eigenvalue $a_r(E')=\nu(r)+\nu^{-1}(r)r$ which is congruent to $a_r$ modulo~$\mathcal{L}$. If else $r$ divides~$c$ (and thus $N$), then $E'$ is a $\mathcal{U}_r$-eigenfunction with corresponding eigenvalue $0=a_r$. Finally, if $r=p_j\in\{p_1,\ldots,p_t\}$, then we have
\begin{equation*}
\left(\mathcal{U}_{p_j}E'\right)(\tau) = \left(\prod_{\substack{i=1 \\ i\not= j}}^t(\mathcal{U}_{p_i}-p_i\nu^{-1}(p_i)\Id)\right)\cdot
\left(\mathcal{U}_{p_j}^2-p_j\nu^{-1}(p_j)\mathcal{U}_{p_j}\right)E(p_1\cdots p_t\tau).
\end{equation*}
Besides, according to \cite[Rk.~3.59]{Shi94}, we have~:
\begin{align*}
 &  \left(\mathcal{U}_{p_j}^2-p_j\nu^{-1}(p_j)\mathcal{U}_{p_j}\right)E(p_1\cdots p_t\tau) \\ 
 =\ &  (\nu(p_j)+\nu^{-1}(p_j)p_j)E(\widehat{p_1\cdots p_t}\tau) - p_jE(p_1\cdots p_t\tau)- p_j\nu^{-1}(p_j)E(\widehat{p_1\cdots p_t}\tau) \\
 =\ & \nu(p_j)\left(\mathcal{U}_{p_j}-p_j\nu^{-1}(p_j)\Id\right)E(p_1\cdots p_t\tau),
\end{align*}
where $\displaystyle{\widehat{p_1\cdots p_t}=\prod_{\substack{i=1 \\ i\not= j}}^tp_i}$. This equality proves that $E'$ is a $\mathcal{U}_{p_j}$-eigenfunction with corresponding eigenvalue~$\nu(p_j)$ and the congruence $\nu(p_j)\equiv a_{p_j}\pmod{\mathcal{L}}$  eventually completes the proof of the lemma.
\end{proof}

    \subsubsection{Constant term at $1/\condnu$ and end of the proof of Theorem~\ref{th:last_case}}
Since $E'$ vanishes at~$\infty$, we compute its constant term at another specific cusp, where it is non-vanishing, namely~$1/\condnu$. Put
\begin{equation*}
\gamma=\begin{pmatrix}
     1 & 0 \\
     \condnu & 1\\                                                                                                                                       \end{pmatrix}\in\SL(2,\Z).
\end{equation*}
We postpone the proof of the following proposition to~\S\ref{sss:proof_prop_cst_term_E2_bar_gamma}.
\begin{prop}\label{prop:cst_term_E2_bar_gamma}
The constant term of the Fourier expansion of~$E'|_2\gamma$ is the non-zero algebraic number in~$\mathcal{O}_L[1/\condnu^2](\mu_{\condnu^2})$~:
\begin{equation*}
\Upsilon'=-\nu(-1)\left(\frac{\condnu}{c_0}\right)^2 
\frac{W\left((\nu^{2})_0\right)}{W(\nu)}
\frac{B_{2,(\nu^{2})_0^{-1}}}{4}
\left(\prod_{i=1}^t\left(1-p_i^{-1}\right)\right)
\cdot\left(\prod_{p\mid \condnu}\left(1-(\nu^2)_0(p)p^{-2}\right)\right),
\end{equation*}
where the second product runs over the primes and $(\nu^2)_0$ is the primitive Dirichlet character associated to $\nu^2$ of modulus~$c_0\mid \condnu$.
\end{prop}

Using this proposition, we now complete the proof of Thm.~\ref{th:last_case}. Let $\Theta$ be the Katz' operator on modular forms over~$\Fellbar$ whose action on $q$-expansions is given by $q\frac{d}{dq}$ (denoted $A\theta$ in~\cite{Kat77}). Assume $\ell>k+1=3$. Lemma~\ref{lem:strong_cong_f_E'} implies that $\Theta(\fbar)=\Theta(\Ebar)$ where $\fbar$ and $\Eprimebar$ are the modular forms over~$\Fellbar$ obtained by reduction modulo~$\mathcal{L}$ of $f$ and $E'$ respectively. Moreover Katz has proved that if $\ell>3$, then $\Theta$ is injective (\cite[Cor.~(3)]{Kat77}). Under this assumption, it thus follows that the Eisenstein series $E'$ becomes cuspidal after reduction.

Put $\epsilon=(\nu^2)_0^{-1}$. By Prop.~\ref{prop:cst_term_E2_bar_gamma} and using the assumption $\condnu=c$, we therefore have that $\ell$ divides the numerator of the norm of~:
\begin{equation*}
\Upsilon'=\pm\left(\frac{c}{c_0}\right)^2 \frac{W\left(\epsilon^{-1}\right)}{W(\nu)}
\frac{B_{2,\epsilon}}{4}
\left(\prod_{i=1}^t\left(1-p_i^{-1}\right)\right)
\cdot\left(\prod_{p\mid c}\left(1-\epsilon^{-1}(p)p^{-2}\right)\right).
\end{equation*}
By Lemma~\ref{lem:Gauss_sum}, the prime divisors of the norm of~$W\left(\epsilon^{-1}\right)/W(\nu)$ divide~$N$ and therefore are co-prime to~$\ell$. The same obviously holds for $c/c_0$. It thus follows that either $p_i\equiv1\pmod{\ell}$ for some $1\le i\le t$ or $\ell$ divides the norm of either~$p^2-\epsilon^{-1}(p)$ for some~$p$ dividing $c$ or the norm of the numerator of~$B_{2,\epsilon}/4$. This completes the proof of Thm.~\ref{th:last_case}.

    \subsubsection{Proof of Proposition~\ref{prop:cst_term_E2_bar_gamma}}\label{sss:proof_prop_cst_term_E2_bar_gamma}

Let us first introduce notation as in the proof of Prop.~\ref{prop:constant_terms_E}. Put~:
\begin{equation*}
G=\frac{C_2W(\nu)}{\condnu^2}E,\quad\textrm{where }C_2=-4\pi^2
\end{equation*}
and similarly
\begin{equation*}
G'=\frac{C_2W(\nu)}{\condnu^2}E'.
\end{equation*}
For simplicity, we shall denote by $\ibar$ the elements of 
\begin{equation*}
\mathcal{N}=\{(i_1,\ldots,i_j)\textrm{ such that }j\in\{1,\ldots,t\}\textrm{ and }1\le i_1<\cdots<i_j\le t\}.
\end{equation*}
If $\ibar=(i_1,\ldots,i_j)\in\mathcal{N}$, we put~:
\begin{equation*}
p_{\ibar}=p_{i_1}\cdots p_{i_j}\quad\textrm{and}\quad a_{\ibar}=a_{p_{i_1}}\cdots a_{p_{i_j}}.
\end{equation*}
Let $v=\overline{(c_v,d_v)}\in(\Z/\condnu^2\Z)^2$ of order~$\condnu^2$. Following~\cite[\S4.6]{DiSh05}, define
\begin{equation}\label{eq:G2v}
G_2^{v}(\tau)=\frac{1}{(c_v\tau+d_v)^2}+\frac{1}{\condnu^4}\sideset{}{'}\sum_{d\in\Z}\frac{1}{\left(\frac{c_v\tau+d_v}{\condnu^2}-d\right)^2}+\frac{1}{\condnu^4}\sum_{c\not=0}\sum_{d\in\Z}\frac{1}{\left(\frac{c_v\tau+d_v}{\condnu^2}-c\tau-d\right)^2}
\end{equation}
where the primed summation notation means to sum over non-zero integers. For any $\ibar\in\mathcal{N}$ and any $v\in(\Z/\condnu^2\Z)^2$ of order~$\condnu^2$, put 
\begin{equation*}
\displaystyle{G_2^{v,p_{\ibar}}(\tau)=G_2^v(p_{\ibar}\tau)}\quad\textrm{and}\quad G^{p_{\ibar}}(\tau)=G(p_{\ibar}\tau).
\end{equation*}
According to~\cite[\S4.2]{DiSh05} and the definition of~$E$ (cf. \S\ref{ss:E}), we have
\begin{equation*}
G=\frac{1}{2}\sum_{i,j,l=0}^{\condnu-1}\nu(ij)G_2^{\overline{(i\condnu,j+l\condnu)}}
\end{equation*}
and therefore
\begin{equation}\label{eq:decomp_Gpi}
G^{p_{\ibar}}=\frac{1}{2}\sum_{i,j,l=0}^{\condnu-1}\nu(ij)G_2^{\overline{(i\condnu,j+l\condnu)},p_{\ibar}}.
\end{equation}

\begin{lemma}\label{lem:cst_term_G2vpi}
Let $v=\overline{(c_v,d_v)}\in(\Z/\condnu^2\Z)^2$ of order~$\condnu^2$. The constant term of $G_2^{v,p_{\ibar}}|_2\gamma$ is 
\begin{equation*}
\Upsilon_{v,\ibar}=\vartheta(\overline{c_vp_{\ibar}+d_v\condnu})\left(\frac{1}{p_{\ibar}}\right)^2\zeta^{\overline{d_v/p_{\ibar}}}(2)
\end{equation*}
where the bar means reduction modulo~$\condnu^2$,
\begin{equation*}
\vartheta(\overline{n})=\displaystyle{\left\{\begin{array}{ll}
                                                    1 & \textrm{if }n\equiv0\pmod{\condnu^2} \\
						    0 & \textrm{otherwise}
                                                    \end{array}
\right.},\quad \zeta^{\overline{n}}(2)=\displaystyle{\sideset{}{'}\sum\limits_{m\equiv n\pmod{\condnu^2}}\frac{1}{m^2}},
\end{equation*}
and the primed summation notation means to sum over non-zero integers. 
\end{lemma}
\begin{proof}
We first compute $G_2^{v,p_{\ibar}}|_2\gamma$ using~(\ref{eq:G2v}). We find~:
\begin{multline*}
\left(G_2^{v,p_{\ibar}}|_2\gamma\right)(\tau)=\frac{1}{(c_vp_{\ibar}\tau+d_v(\condnu\tau+1))^2}
+\sideset{}{'}\sum_{d\in\Z}\frac{1}{(c_vp_{\ibar}\tau+d_v(\condnu\tau+1)-\condnu^2d(\condnu\tau+1))^2}\\
+\sum_{c\not=0}\sum_{d\in\Z}\frac{1}{(c_vp_{\ibar}\tau+d_v(\condnu\tau+1)-\condnu^2(cp_{\ibar}\tau+d(\condnu\tau+1)))^2}.
\end{multline*}
In other words, we have $\left(G_2^{v,p_{\ibar}}|_2\gamma\right)(\tau)=A+B$, where 
\begin{equation*}
A=\frac{1}{((c_vp_{\ibar}+d_v\condnu)\tau+d_v)^2}
+\sideset{}{'}\sum_{d\in\Z}\frac{1}{((c_vp_{\ibar}+d_v\condnu-\condnu^2d\condnu)\tau+d_v-\condnu^2d)^2}
\end{equation*}
and 
\begin{equation*}
B=\sum_{c\not=0}\sum_{d\in\Z}\frac{1}{((c_vp_{\ibar}+d_v\condnu-\condnu^2(cp_{\ibar}+d\condnu))\tau+d_v-\condnu^2d)^2}.
\end{equation*}
Since $\gcd(p_{\ibar},\condnu)=1$, we may assume without loss of generality that $0\le c_vp_{\ibar}+d_v\condnu<\condnu^2$. Therefore the constant term of~$A$ is given by~:
\begin{equation*}
\vartheta(\overline{c_vp_{\ibar}+d_v\condnu})\frac{1}{d_v^2}
\end{equation*}
and the one of~$B$ by~:
\begin{equation*}
\vartheta(\overline{c_vp_{\ibar}+d_v\condnu})\sum_{c\not=0}\sum_{\substack{d\in\Z\\ cp_{\ibar}+d\condnu=0}}\frac{1}{(d_v-\condnu^2d)^2}.
\end{equation*}
Therefore, the constant term of $G_2^{v,p_{\ibar}}|_2\gamma$ is~:
\begin{equation*}
\Upsilon_{v,\ibar}=\vartheta(\overline{c_vp_{\ibar}+d_v\condnu})\sum_{c\in\Z}\sum_{\substack{d\in\Z\\ cp_{\ibar}+d\condnu=0}}\frac{1}{(d_v-\condnu^2d)^2}.
\end{equation*}
Note that if $\vartheta(\overline{c_vp_{\ibar}+d_v\condnu})=1$, then $d_v\not\equiv0\pmod{\condnu^2}$ since $v$ is of order~$\condnu^2$. A change of variable yields~:
\begin{equation*}
\Upsilon_{v,\ibar}=\vartheta(\overline{c_vp_{\ibar}+d_v\condnu})\sum_{c\in\Z}\sum_{\substack{d\in\Z\\ cp_{\ibar}+d\condnu=0\\ (c,d)\equiv v\ (\condnu^2)}}\frac{1}{d^2}
\end{equation*}
and thus
\begin{equation*}
\Upsilon_{v,\ibar}=\vartheta(\overline{c_vp_{\ibar}+d_v\condnu})\sum_{\substack{d\not=0 \\ d\equiv d_v\ (\condnu^2)\\ p_{\ibar}\mid d}}\frac{1}{d^2}
=\vartheta(\overline{c_vp_{\ibar}+d_v\condnu})\sum_{\substack{m\not=0 \\ m\equiv d_v/p_{\ibar}\ (\condnu^2)}}\frac{1}{(p_{\ibar}m)^2}.
\end{equation*}
Finally we get $\Upsilon_{v,\ibar}=\vartheta(\overline{c_vp_{\ibar}+d_v\condnu})/p_{\ibar}^2\cdot\zeta^{\overline{d_v/p_{\ibar}}}(2)$ as asserted.
\end{proof}

Using this lemma and formula~(\ref{eq:decomp_Gpi}), we are now able to compute the constant term of~$G^{p_{\ibar}}|_2\gamma$.
\begin{lemma}\label{lem:cst_term_G2pi}
The constant term of $G^{p_{\ibar}}|_2\gamma$ is 
\begin{equation*}
\Upsilon_{\ibar}=\nu(p_{\ibar})\frac{1}{p_{\ibar}^2}\cdot\Upsilon_0,\quad\textrm{with }
\Upsilon_0=-\nu(-1)W\left((\nu^{2})_0\right)\frac{C_2}{c_0^2}\frac{B_{2,(\nu^{2})_0^{-1}}}{4}
\prod_{p\mid \condnu}\left(1-(\nu^2)_0(p)p^{-2}\right),
\end{equation*}
where $(\nu^2)_0$ is the primitive Dirichlet character associated to $\nu^2$ of modulus~$c_0\mid \condnu$.
\end{lemma}
\begin{proof}
The proof of this lemma is quite similar to the proof of~Prop.~\ref{prop:constant_terms_E}. According to~(\ref{eq:decomp_Gpi}), we have~:
\begin{equation*}
\Upsilon_{\ibar}=\frac{1}{2}\sum_{i,j,l=0}^{\condnu-1}\nu(ij)\Upsilon_{\overline{(i\condnu,j+l\condnu)},\ibar}
\end{equation*}
and thus by Lemma~\ref{lem:cst_term_G2vpi}~:
\begin{equation*}
\Upsilon_{\ibar}=\frac{1}{2}\cdot\frac{1}{p_{\ibar}^2}
\sum_{i,j,l=0}^{\condnu-1}\nu(ij)\vartheta(\overline{i\condnu p_{\ibar}+\condnu(j+l\condnu)})
\zeta^{\overline{d_v/p_{\ibar}}}(2).
\end{equation*}
This yields to~:
\begin{align*}
\Upsilon_{\ibar} & = \frac{1}{2}\cdot\frac{1}{p_{\ibar}^2}
\sum_{l=0}^{\condnu-1}\sum_{\substack{j=0\\ \gcd(j,\condnu)=1}}^{\condnu-1}
\nu\left(-\frac{j^2}{p_{\ibar}}\right)
\zeta^{\overline{d_v/p_{\ibar}}}(2)\\
 & =  \frac{1}{2}\cdot\frac{1}{p_{\ibar}^2}\nu(p_{\ibar})\nu(-1)
 \sum_{l=0}^{\condnu-1}\sum_{\substack{j=0\\ \gcd(j,\condnu)=1}}^{\condnu-1} 
\nu\left((j^2/p_{\ibar})^2\right)
\sideset{}{'}\sum_{m\equiv (j+l\condnu)/p_{\ibar}\ (\condnu^2)}\frac{1}{m^2}\\
 & = \frac{1}{p_{\ibar}^2}\nu(p_{\ibar})\nu(-1)L(2,\nu^2).\\
\end{align*}
Let $(\nu^2)_0$ be the primitive character associated to $\nu^2$ of modulus~$c_0\mid \condnu$. We have~:
\begin{equation*}
L(2,\nu^2)=L\left(2,(\nu^2)_0\right)\prod_{p\mid \condnu}\left(1-(\nu^2)_0(p)p^{-2}\right).
\end{equation*}
Applying Prop.~\ref{prop:FE_for_L_functions} to $\psi=(\nu^2)_0$ and $m=k$, we get~:
\begin{equation*}
L(2,(\nu^2)_0)=-W\left((\nu^{2})_0\right)\frac{C_2}{c_0^2}\frac{B_{2,(\nu^{2})_0^{-1}}}{4}\not=0
\end{equation*}
and thus
\begin{equation*}
\Upsilon_{\ibar}=-\frac{1}{p_{\ibar}^2}\nu(p_{\ibar})\nu(-1)W\left((\nu^{2})_0\right)\frac{C_2}{c_0^2}
\frac{B_{2,(\nu^{2})_0^{-1}}}{4}
\prod_{p\mid \condnu}\left(1-(\nu^2)_0(p)p^{-2}\right),
\end{equation*}
as claimed.
\end{proof}

Let us now complete the proof of Prop.~\ref{prop:cst_term_E2_bar_gamma}. With the notation introduced at the beginning of this paragraph and Eq.~(\ref{eq:E'_expanded_form}), we have~:
\begin{equation*}
G'|_2\gamma=G|_2\gamma+\sum_{\ibar\in\mathcal{N}}(-1)^{\sharp\ibar}
p_{\ibar}\nu^{-1}(p_{\ibar})G^{p_{\ibar}}|_2\gamma.
\end{equation*}
Therefore, according to Prop.~\ref{prop:constant_terms_E} and Lem.~\ref{lem:cst_term_G2pi}, the constant term of~$G'|_2\gamma$ is~:
\begin{equation*}
\Upsilon_0+\sum_{\ibar\in\mathcal{N}}(-1)^{\sharp\ibar}
p_{\ibar}\nu^{-1}(p_{\ibar})\Upsilon_{\ibar}
=\Upsilon_0\left(1+\sum_{\ibar\in\mathcal{N}}(-1)^{\sharp\ibar}
p_{\ibar}\nu^{-1}(p_{\ibar})\nu(p_{\ibar})\frac{1}{p_{\ibar}^2}\right)
=\Upsilon_0\prod_{i=1}^t\left(1-p_i^{-1}\right)
\end{equation*}
where $(\nu^2)_0$ is the primitive character associated to $\nu^2$ of modulus~$c_0\mid \condnu$. Prop.~\ref{prop:cst_term_E2_bar_gamma} now follows from the normalization $E'=(\condnu^2/(C_2W(\nu)))G'$.

    \section{Dihedral representations}

    \subsection{Preliminaries: twisting and CM forms}

Let $M$ be an integer, $F(\tau)=\sum_{n\ge 1}a_n(F)q^n \in\ModSpace[S][k][0][M]$ and $\psi$ be a Dirichlet character of modulus~$f\ge 1$. Define~:
\begin{equation*}
(F\otimes\psi)(\tau)=\sum_{n\ge 1}a_n(F)\psi(n)q^n .
\end{equation*}
The following result is a special case of~\cite[Prop.~3.64]{Shi94}.
\begin{lemma}\label{lemma:twist}
With the notations above, assume $\psi$ to be a qua\-dratic primitive Dirichlet character. Then $F\otimes\psi$ belongs to~$\ModSpace[S][k][0][\lcm(M,f^2)]$. Moreover, if $F$ is a normalized Hecke eigenform for the Hecke operators~$\{T_p\}_{p\nmid M}$, then $F\otimes\psi$ is a normalized Hecke eigenform for the Hecke operators~$\{T_p\}_{p\nmid fM}$ with corres\-ponding eigenvalues~$\{a_p(F)\psi(p)\}_{p\nmid fM}$.
\end{lemma}
We take the following definition for CM forms (\cite{Rib77}).
\begin{defi}[CM forms]\label{defi:cm_forms} 
Assume that $\psi$ is not the trivial character. The form $F$ has complex multiplication (or, $F$ is a CM form) by~$\psi$ if $a_p(F)=a_p(F)\psi(p)$ for all~$p$ in a set of primes of density~$1$.
\end{defi}

    \subsection{Statement of the result}

Recall that
\begin{equation*}
\projrhobar: \GQ\stackrel{\rhobar_{f,\lambda}}{\longrightarrow}\GL(2,\Flambda)\longrightarrow \PGL(2,\Flambda),
\end{equation*}
where $\GQ=\Gal(\Qbar/\Q)$ and put $\projim=\projrhobar(\GQ)$.

The following result is a generalization to arbitrary weights and fields of coefficients of a theorem on the surjectivity of Galois representations attached to elliptic curves over~$\Q$ independently proved by Kraus (\cite{Kra95}) and Cojocaru (\cite{Coj05}). In particular, it implies that in the case of dihedral projective image, $\ell$ is explicitly bounded in terms of~$k$ and~$N$.

\begin{thm}\label{thm:dihedral_case}
Assume that Assume $\projim$ dihedral. If $f$ does not have complex multiplication, then we have
\begin{equation*}
\ell\le \left(2\left(4.8kN^2(1+\log\log N)\right)^{\frac{k-1}{2}}\right)^{[K:\Q]}.
\end{equation*}
Besides, if $N$ is square-free, then either $\ell\mid N$, or $\ell\le k$, or $\ell=2k-1$.
\end{thm}

\noindent\textit{Remarks.}
\begin{enumerate}
\item The integer $[K:\Q]$ is bounded from above by the dimension~$g_0^{\sharp}(k,N)$ of the new subspace of~$\ModSpace$. A closed formula in terms of $k$ and $N$ for $g_0^{\sharp}(k,N)$ as well as asymptotic estimates can be found in~\cite{Mar05}.
\item When $N=1$, the result goes back to Ribet (see the proof of~(ii) p.~264 and the remark after Cor.~4.5 of~\cite{Rib75}). Moreover, our argument for the case of arbitrary square-free level is a combination of tricks from~\cite{Rib85} and~\cite{Rib97}.
\item A newform of square-free level and trivial Nebentypus is automatically non-CM (see e.g.~\cite[\S4]{Tsa12}).
\end{enumerate}

  \subsection{Proof of Theorem~\ref{thm:dihedral_case}}
Assume $\ell\nmid N$ and $\projim$ dihedral. Then $\projim$ is an extension of $\{\pm1\}$ by a cyclic group~$C$ and every element of~$\Glambdabar$ which does not map to~$C$ has trace~$0$. Hence, we may consider the following quadratic character~:
\begin{equation*}
\elambda:\GQ\stackrel{\projrhobar}{\longrightarrow}\projim\rightarrow\{\pm1\}.
\end{equation*}
Let $L_{\lambda}$ be the number field cut out by $\projrhobar$ and $K_{\lambda}/\Q$ its quadratic sub-extension fixed by the kernel of~$\elambda$. The extension $L_{\lambda}/\Q$ has Galois group isomorphic to~$\projim$ while $C\simeq\Gal(L_{\lambda}/K_{\lambda})$.
Clearly, $\elambda$ is unramified outside~$\ell N$. The following proposition describes more precisely the ramification set of~$\elambda$.
\begin{prop}
Assume $\ell\nmid N$.
\begin{enumerate}
\item Let $p\not=\ell$ be a ramified prime for~$\elambda$. Then $p^2\mid N$.
\item Assume $\ell>k$ and
\begin{enumerate}
\item either $f$ is ordinary at~$\lambda$ and $\ell\not=2k-1$;
\item or $f$ is not ordinary at~$\lambda$ and $\ell\not=2k-3$.
\end{enumerate}
Then, $\elambda$ is unramified at~$\ell$.
\end{enumerate}
\end{prop}
\begin{proof}
Let $p$ be a prime dividing~$N$ exactly once. By~\S\ref{ss:Steinberg}, we know that the inertia subgroup~$I_p$ at~$p$ acts unipotently in~$\rhobar$. Since $\Glambdabar$ has prime-to-$\ell$ order, it follows that $I_p$ acts trivially. So $\rhobar$ and, hence, $\elambda$ are unramified at~$p$. This proves the first part of the proposition.

Assume now $\ell>k$. Let $I_{\ell}$ be the inertia group of a decomposition subgroup at~$\ell$ and recall that $\ell\nmid N$. We prove that $\elambda$ is unramified at~$\ell$ under conditions~(a) and~(b) in turn.
\begin{enumerate}
\item[(a)] Assume that $f$ is ordinary at~$\lambda$ and $\ell\not=2k-1$. By~\S\ref{ss:local_description_at_ell}, we have
\begin{equation*}
\rhobar_{|I_{\ell}}\simeq\begin{pmatrix}
                         \cyclomod^{k-1} & \star \\
			    0 & 1\\
                         \end{pmatrix}.
\end{equation*}
But $\Glambdabar$ has prime-to-$\ell$ order and therefore $\star=0$. In particular, $\projrhobar(I_{\ell})$ is isomorphic to the image of~$\cyclomod^{k-1}$ which is, by Lemma~\ref{lemma:cardinality_image}, cyclic of order~$(\ell-1)/\gcd(\ell-1,k-1)>2$. Therefore, it has to be included in~$C$ and hence $\elambda$ is unramified at~$\ell$.
\item[(b)] Assume that $f$ is not ordinary at~$\lambda$ and $\ell\not=2k-3$. By~\S\ref{ss:local_description_at_ell}, $\projrhobar(I_{\ell})$ is isomorphic to the image of~$I_{\ell}$ under~$\psi^{(\ell-1)(k-1)}$ where $\psi$ is a fundamental character of level~$2$. By the assumption $\ell\not=2k-3$ and Lemma~\ref{lemma:cardinality_image}, it is therefore cyclic of order~$(\ell+1)/\gcd(\ell+1,k-1)>2$. We conclude as before.
\end{enumerate}
\end{proof}
Assume $N$ to be square-free and $\ell>k$. Then, by the above proposition, $K_{\lambda}$ is the unique quadratic extension of~$\Q$ ramified at~$\ell$ only and $\ell\in\{2k-1,2k-3\}$. The case $\ell=2k-3$ however does not occur. This is proved in~\cite[Lem.~3.2]{Die12}. Hence Thm.~\ref{thm:dihedral_case} in the square-free level case.

Assume now that $N$ is any integer not divisible by~$\ell$, and that $\ell>k$ satisfies $\ell\not=2k-1$ and $\ell\not=2k-3$. We may identify~$\elambda$ with a Dirichlet character. Let us denote by $\condeps$ its conductor. It is co-prime to~$\ell$ by the above proposition. We then have $\condeps=|D_{K_{\lambda}}|$ where $D_{K_{\lambda}}$ is the fundamental discriminant of the quadratic field~$K_{\lambda}$ fixed by the kernel of~$\elambda$ (\cite[VII.~\S11]{Neu99}). In particular, if $K_{\lambda}=\Q(\sqrt{D_0})$ with $D_0$ square-free, then $\condeps=D_0$ or $4D_0$ depending on whether $D_0\equiv 1\pmod{4}$ or not. If moreover, $\ell>2k-1$, then by the proposition above, $\condeps^2\mid 2^4 N$. Put $g=f\otimes\elambda$. By the Lemma~\ref{lemma:twist}, $g\in\ModSpace[S][k][0][2^4N]$ and for any prime $p\nmid2N$, $g$ is an eigenform for the $T_p$ Hecke operator with corresponding eigenvalue~$a_p(g)=a_p\elambda(p)$. Let $D_0'=\varepsilon\prod_{3\leq p\mid N}p$ be the product of all odd primes dividing~$N$ with a sign $\varepsilon\in\{\pm1\}$ chosen so that $D_0'\equiv 3\pmod{4}$. Then $4D_0'$ is a fundamental discriminant and the Kronecker symbol $\psi=\left(4D_0'/\cdot\right)$ is a primitive quadratic Dirichlet character of modulus $4D_0'$ (\cite[Th.~2.2.15]{Coh07}) precisely ramified at the primes dividing~$2N$. Put~:
\begin{equation*}
\ftilde=f\otimes\psi\quad\textrm{and}\quad \gtilde=g\otimes\psi.
\end{equation*}
Since $(4D_0')^2\mid 2^4N^2$, it follows from Lemma~\ref{lemma:twist} that $\ftilde,\gtilde\in\ModSpace[S][k][0][2^4N^2]$ and for any integer~$n$, we have~:
\begin{equation}\label{eq:Fourier_coeff_ftilde_gtilde}
\left\{\begin{array}{rcl}
       a_n(\ftilde) & = & a_n\psi(n) \\
       a_n(\gtilde) & = & a_n\elambda(n)\psi(n).
       \end{array}
\right.
\end{equation}
Since $f$ is assumed to be non-CM (in the sense of Def.~\ref{defi:cm_forms}), we have $\ftilde\not=\gtilde$ and by~\cite[Th.~1]{Mur97}, there exists an integer
\begin{equation}\label{eq:bound_for_n}
n\leq \frac{4k}{3}N^2\prod_{p\mid 2N}\left(1+\frac{1}{p}\right)\le 2kN^2\prod_{p\mid N}\left(1+\frac{1}{p}\right)
\end{equation}
such that $a_n(\ftilde)\not=a_n(\gtilde)$. According to~(\ref{eq:Fourier_coeff_ftilde_gtilde}), it follows that  we have~:
\begin{equation*}
\psi(n)\not=0,\quad a_n\not=0\quad\textrm{and}\quad \elambda(n)=-1.
\end{equation*}
From the condition~$\elambda(n)=-1$, we deduce that there exists a prime divisor~$q$ of~$n$ together with an odd integer~$t$ such that $q^t\mid n$ but $q^{t+1}\nmid n$ and $\elambda(q)=-1$. If $q=\ell$, we are done in bounding $\ell$ in terms of~$k$ and~$N$. Assume therefore $q\not=\ell$. The multiplicativity of the Fourier coefficients of~$f$ gives that $a_{q^t}\mid a_n$ and hence (since $t$ is odd) that $a_q\not=0$. Besides, since $\elambda(q)=-1$, the image under~$\rhobar_{f,\lambda}$ of a Frobenius at~$q$ has trace~$0$ modulo~$\lambda$. In other words, $\ell$ divides the norm of the non-zero algebraic integer~$a_q$. Applying Deligne's estimate on the Fourier coefficients of~$f$ and its Galois conjugates by $\Qbar$-automorphisms, we get that~:
\begin{equation}\label{eq:bound_Deligne}
\ell \leq \mathrm{N}_{K/\Q}\left(a_q\right)=\prod_{\sigma: K\hookrightarrow \C} \left|\sigma(a_q)\right|\leq \left(2q^{(k-1)/2}\right)^{[K:\Q]}.
\end{equation}
Besides, using \cite[(3.27)]{RoSc62} and inequality~(\ref{eq:bound_for_n}), we get the following estimate for~$q$~:
\begin{equation}\label{eq:bound_Rosser}
q\le 4.8kN^2(1+\log\log N).
\end{equation}
The theorem follows from~(\ref{eq:bound_Deligne}) and~(\ref{eq:bound_Rosser}).

    \section{Projective image isomorphic to $A_4$, $S_4$ or $A_5$}

The following result is proved in a different way in~\cite{Rib85}.
\begin{thm}\label{thm:image_isom_A_4}
If $\projim$ is isomorphic to $A_4$, $S_4$ or $A_5$, then either $\ell\mid N$ or $\ell\le 4k-3$. 
\end{thm}
\begin{proof}
Assume that $\ell\nmid N$ and $\ell>k$. Then, by~\S\ref{ss:local_description_at_ell}, $\projim$ has a cyclic subgroup given the image of inertia at~$\ell$. In the case of ordinariness, this cyclic subgroup is isomorphic to the image of~$\cyclomod^{k-1}$ which has order~$>5$ if $\ell>4k-3$ by Lemma~\ref{lemma:cardinality_image}. If else $f$ is not ordinary at~$\lambda$, then it has order $(\ell+1)/\gcd(\ell+1,k-1)$ which is also~$>5$ if $\ell>4k-3$. 

In any case, if $\ell>4k-3$, then $\projim$ has an element of order~$>5$. This rules out the possibility for $\projim$ to be isomorphic to $A_4$, $S_4$ or $A_5$.
\end{proof}

    \section{Numerical examples}

In this section we give some examples illustrating the theorems of the paperµ. All the computations were performed on SAGE~(\cite{sage}).

    \subsection{Reducible representations}

Before dealing with examples, let us first recall that for the representations $\rhobar_{f,\lambda}$, irreducibility is equivalent to absolute irreducibility.

    \subsubsection{Square level case}
Fix $(k,N)=(6,81)$. The new subspace in~$\ModSpace[S][6][0][81]$ is~$18$-dimensional and splits into~$5$ Galois conjugacy classes labeled 81.6a,\dots,81.6e in~SAGE (\cite{sage}). 
According to Theorem~\ref{thm:general_thm}, the prime ideals~$\lambda$ such that $\rhobar_{f,\lambda}$ is reducible for some newform $f\in\ModSpace[S][6][0][81]$ have residue characteristic~$\ell$ in~$\{2,3,5,7,43,1171\}$. Let us first show that $2$, $3$, $7$, $43$ and $1171$ are indeed the residue characteristics of some prime ideals~$\lambda$ for which $\rhobar_{f,\lambda}$ is reducible for the specific (up to Galois conjugacy) modular form~$f$ labeled 81.6c. We have~:
\begin{equation*}
f(\tau)=q + \alpha q^{2} + (\alpha ^{2} - 32)q^{4} + \left(-\frac{1}{4}\alpha ^{3} - \frac{9}{4}\alpha ^{2} + \frac{25}{2}\alpha  + 54\right)q^{5} + O(q^{5}),
\end{equation*}
where $\alpha$ is a root of $X^4 + 3X^3 - 84X^2 - 72X+ 792$.

Let us denote by $K$ the number field generated by~$\alpha$. We call $\nu$ the primitive Dirichlet character modulo~$9$ sending~$2$ on~$\zeta_3$, where $\zeta_3$ is a primitive third root of unity and $L=\Q(\zeta_3)$. Since $\nu$ has order~$3$, we have $\epsilon=\nu$ with the notation of Thm.~\ref{thm:general_thm}. Moreover we have $B_{6,\nu}/12=(751\zeta_3+1172)/3$ which has norm~$3^{-1}\cdot 7\cdot 43\cdot 1171$.

Then we more precisely show that for each $\ell\in\{2,3,7,43,1171\}$ there are prime ideals $\lambda_{\ell}$ and $\mathfrak{p}_{\ell}$ above~$\ell$ in $\mathcal{O}$ and~$\Z[\zeta_3]$ respectively such that $\rhobarss_{f,\lambda_{\ell}}\simeq\rhobar_{E,\mathfrak{p}_{\ell}}$ where $E$ is the following Eisenstein series
\begin{equation*}
E(\tau)=\sum_{n\ge1}\sigma_5^{\nu}(n)q^n=q-(31\zeta_3+32)q^2+(1023\zeta_3+31)q^4+(3124\zeta_3-1)q^5+O(q^5).
\end{equation*}
Such an isomorphism is proved to hold by checking that for all integers $n$ up to the Sturm bound (which, here, equals~$54$) we have a congruence
\begin{equation*}
a_n\equiv a_n(E)\pmod{\mathcal{L}_{\ell}},
\end{equation*}
for some prime ideal~$\mathcal{L}_{\ell}$ above~$\ell$ in the integer ring of the compositum~$KL$. For instance, if $\ell=43$, we can take
\begin{equation*}
\mathcal{L}_{43}=\left(43, \alpha + \zeta_3 - 6 \right).
\end{equation*}
Therefore we have $\rhobarss_{f,\lambda_{\ell}}\simeq\nubar_{\ell}\oplus\nubar_{\ell}^{-1}\cyclomod^5$ where
\begin{equation*}
\nubar_{\ell}~:\GQ\twoheadrightarrow(\Z/9\Z)^{\times}\stackrel{\nu}{\rightarrow}\Z[\zeta_3]\twoheadrightarrow\Z[\zeta_3]/\mathfrak{p}_{\ell}
\end{equation*}
is $\nu$ modulo~$\mathfrak{p}_{\ell}$ viewed as a character of~$\GQ$. For each $\ell$ as above the corresponding ideals~$\lambda_{\ell}$ and $\mathfrak{p}_{\ell}$ are listed in Table~\ref{table:congruence_primes} (as given in SAGE).

\begin{table}[h!]
\centering
\begin{tabular}{c|c|c}
$\ell$ & $\lambda_{\ell}$ & $\mathfrak{p}_{\ell}$ \\
\hline
$2$ & $\left(2, \alpha^3/36 + \alpha^2/4 - 7\alpha/6 - 7\right)$ & $(2)$ \\
$3$ & $\left(3, -\alpha^3/36 + \alpha^2/12 + 7\alpha/6 - 7\right)$ & $(2\zeta_3 +
1)$ \\
$7$ & $\left(7, \alpha^3/36 + \alpha^2/12 - 5\alpha/3 + 2\right)$ & $(3\zeta_3 + 1)$ \\
$43$ & $\left(43,\alpha^3/36 + \alpha^2/12 - 5\alpha/3 - 20\right)$& $(7\zeta_3+6)$ \\
$1171$ & $\left(1171, \alpha^3/36 + \alpha^2/12 - 5\alpha/3 - 586\right)$ & $(39\zeta_3 + 25)$ \\
\end{tabular}
\caption{Congruence primes for~$f$ and~$E$}
\label{table:congruence_primes}
\end{table}
Let us now see what happens for the remaining prime, namely~$\ell=5$. For the specific newform above with coefficients field~$K$, we have $5\mathcal{O}=\lambda_5\lambda_5'$ where $\lambda_5=(5,\alpha+4)$ and $\lambda_5'=(5,\alpha^3+4\alpha^2+3)$. Then $\lambda_5$ and $\lambda_5'$ have inertia degree $1$ and $3$ respectively. Besides, if $\Frob_2$ denotes a Frobenius at~$2$, the characteristic polynomial of $\rhobar_{f,\lambda_5}(\Frob_2)$ and $\rhobar_{f,\lambda_5'}(\Frob_2)$ is $X^2-\alpha X+2^5$. Such a polynomial being irreducible modulo~$\lambda_5$ and $\lambda_5'$ as one checks, we get that $\rhobar_{f,\lambda_5}$ and~$\rhobar_{f,\lambda_5'}$ are both irreducible.

For each pair $(f,\lambda)$ where $f$ is a newform in $\ModSpace[S][6][0][81]$ and $\lambda$ is a prime ideal in~$\mathcal{O}$ above~$5$ we give in Table~\ref{table:primes_p} the smallest prime number~$p\not=3,5$ and~$\le 100$ for which the characteristic polynomial of~$\rhobar_{f,\lambda}(\Frob_p)$ is irreducible. 
\begin{table}[h!]
\centering
\begin{tabular}{c|c|c|c}
$f$ & $K=\Q(\alpha)$ & $\lambda$ &  $p$ \\
\hline
\multirow{2}{*}{81.6a} & \multirow{2}{*}{$\alpha^2 + 3\alpha - 30=0$} & $(-6\alpha+25)$ & $2$ \\
& &  $(-6\alpha-43)$ & $7$ \\
\hline
\multirow{2}{*}{81.6b} & \multirow{2}{*}{$\alpha^2 - 3\alpha - 30=0$} & $(-6\alpha-25)$ & $2$ \\
& &  $(-6\alpha+43)$ & $7$ \\
\hline
\multirow{2}{*}{81.6c} & \multirow{2}{*}{$\alpha^4 + 3\alpha^3 - 84\alpha^2 - 72\alpha+ 792=0$} & $(5,\alpha+4)$ & $2$ \\
& &  $(5,\alpha^3+4\alpha^2+3)$ & $2$ \\
\hline
\multirow{2}{*}{81.6d} & \multirow{2}{*}{$\alpha^4 - 3\alpha^3 - 84\alpha^2 + 72\alpha+ 792=0$} & $(5,\alpha+1)$ & $2$ \\
& &  $(5,\alpha^3+\alpha^2+2)$ & $2$ \\
\hline
\multirow{3}{*}{81.6e} & \multirow{3}{*}{$\alpha^6 - 171\alpha^4 + 7128\alpha^2 -432=0$} & $(5,\alpha^2+1)$ & $\emptyset$ \\
& &  $(5,\alpha^2+3\alpha+3)$ & $7$ \\
& &  $(5,\alpha^2+2\alpha+3)$ & $7$ \\
\end{tabular}
\caption{Smallest prime $p\not=3,5$ and $\le 100$ such that $\rhobar_{f,\lambda}(\Frob_p)$ acts irreducibly}
\label{table:primes_p}
\end{table}
Therefore all the representations $\rhobar_{f,\lambda}$ are irreducible unless perhaps if 
$f$ is the form~81.6e and $\lambda=(5,\alpha^2+1)$. But this latter representation is also proved to be irreducible by noticing that the eigenvalues of $\rhobar_{f,\lambda}(\Frob_2)$ and $\rhobar_{f,\lambda}(\Frob_{19})$ in $\F_{\lambda}$ are $\{3\beta,3\beta\}$ and $\{2\beta+1,3\beta+1\}$ respectively where $\beta$ is the image of~$\alpha$ in~$\F_{\lambda}$ (since if it were reducible, we would have $\rhobarss_{f,\lambda}\simeq\epsilon_1\oplus\epsilon_2$ where both $\epsilon_1$ and $\epsilon_2$ factor through $(\Z/45\Z)^{\times}$). This eventually proves the following proposition.
\begin{prop}
Let $(k,N)=(6,81)$. Then there exists a newform $f\in\ModSpace[S][6][0][81]$ together with a prime ideal~$\lambda$ in~$\mathcal{O}$ such that $\rhobar_{f,\lambda}$ is reducible if and only if $\ell$ belongs to~$\{2,3,7,43,1171\}$.
\end{prop}

    \subsubsection{Square-free level case}

Fix $(k,N)=(4,11)$. The new subspace in~$\ModSpace[S][4][0][11]$ is~$2$-dimensional and generated by one Galois orbit labeled 11.4a in~SAGE (\cite{sage}). Let $f$ be a representative of this Galois orbit. We have
\begin{equation*}
f(\tau)=q+\alpha q^2+(-4\alpha+3)q^3+(2\alpha-6)q^4+(8\alpha-7)q^5+O(q^5),
\end{equation*}
where $\alpha$ is a root of~$X^2-2X-2$. The field $K=\Q(\alpha)$ is the coefficients field of~$f$. According to Theorem~\ref{thm:square_free_level_case}, if $\rhobar_{f,\lambda}$ is reducible then $\lambda$ has residue characteristic~$\ell$ in the set~$\{2,3,5,11,61\}$. For each prime~$\ell$ in~$\{2,3,5,11,61\}$ we give in Table~\ref{table:primes_p_bis} the smallest prime $p\not= 11,\ell$ and $p\le 100$ such that the characteristic polynomial of~$\rhobar_{f,\lambda}(\Frob_p)$ is irreducible. 

\begin{table}[h!]
\centering
\begin{tabular}{c|c|c|c|c|c|c|c}
$\ell$ & $2$ & $3$ & $5$ & \multicolumn{2}{c|}{$11$} & \multicolumn{2}{c}{$61$} \\
\hline
$\lambda$  &  $(\alpha)$ & $(\alpha-1)$ & $(5)$ & $(2\alpha-3)$ & $(2\alpha-1)$ & $(\alpha-9)$ & $(\alpha+7)$ \\  
\hline
$p$ & $3$ & $2$ & $2$ & $2$ & $\emptyset$ & $\emptyset$ & $2$\\ 
\end{tabular}
\caption{Smallest prime $p\not=11,\ell$ and $\le 100$ such that $\rhobar_{f,\lambda}(\Frob_p)$ acts irreducibly}
\label{table:primes_p_bis}
\end{table}
Therefore all such Galois representations are irreducible except perhaps $\rhobar_{f,(2\alpha-1)}$ and $\rhobar_{f,(\alpha-9)}$. These latter representations turn out to be reducible and we have
\begin{equation*}
\rhobarss_{f,(2\alpha-1)}\simeq\cyclomod[11]\oplus\cyclomod[11]^2\quad\textrm{and}\quad 
\rhobarss_{f,(\alpha-9)}\simeq\mathbf{1}\oplus\cyclomod[61]^3\simeq\rhobar_{E_4,61}.
\end{equation*}
This eventually proves the following proposition.
\begin{prop}
Let $(k,N)=(4,11)$. Then there exists a newform $f\in\ModSpace[S][4][0][11]$ together with a prime ideal~$\lambda$ in~$\mathcal{O}$ such that $\rhobar_{f,\lambda}$ is reducible if and only if $\ell=11$ or $\ell=61$.
\end{prop}

    \subsection{Dihedral representation}

In this section we discuss an example of dihedral projective representation attached to some specific newform. The new subspace in~$\ModSpace[S][2][0][1888]$ has dimension~$58$ and is split into $16$ Galois orbits. Among them let us consider the newform~$f$ (up to Galois conjugacy) labeled 1888.10a whose first terms in its Fourier expansion at infinity are
\begin{equation*}
f(\tau)=q + \frac{1}{2}\alpha q^3 + \left(-\frac{1}{16}\alpha^4 + \frac{3}{2}\alpha^2 - \alpha - 2\right)q^5 + O(q^6)
\end{equation*}
where $\alpha$ is a root of~$X^5 + 6X^4 - 20X^3 -128X^2 + 48X + 320$. The prime~$5$ is definitely smaller than the bound given in Thm.~\ref{thm:dihedral_case} (namely $3476092007703911714679$ in this case) and one proves that there is mod.~$5$ representation attached to $f$ which has dihedral projective image. Namely, let us consider the prime ideal $\lambda=(5,\alpha/2)$ above~$5$ in~$\mathcal{O}$. Then one checks that the representation $\rhobar_{f,\lambda}$ is isomorphic to $\rhobar_{\mathcal{E},5}$ where $\mathcal{E}$ is the rational CM elliptic curve of conductor~$32$ given by the equation $y^2=x^3-x$. Since $5\equiv 1\pmod{4}$, one knows by the theory of complex multiplication that $\rhobar_{\mathcal{E},5}$ has image included in the normalizer of a split Cartan subgroup of~$\GL(2,\F_5)$. The same conclusion for $\rhobar_{f,\lambda}$ thus follows.

    \subsection{Projective image isomorphic to $A_4$, $S_4$ or $A_5$}

As an illustration of Thm.~\ref{thm:image_isom_A_4}, we report here on an example due to Ribet~(\cite[Rk.~2, p.~283]{Rib97}) and recalled in~\cite[Ex.~3.2, p.~244]{KiVe05} (we warn the reader that the term ``exceptional'' therein refers to a modular representation with projective image isomorphic to $A_4$, $S_4$ or $A_5$). The new subspace in~$\ModSpace[S][2][0][23]$ is $2$-dimensional and generated by one Galois orbit labeled 23.4a in SAGE with coefficients field $K=\Q(\alpha)$ where $\alpha$ is a root of~$X^2+X-1$. Let $\lambda$ be the unique prime ideal above~$3$ in~$\mathcal{O}$. It is shown in~{\em loc. cit.} that the corresponding projective representation has image isomorphic to~$A_5$ and that the field cut out by its kernel is the $A_5$-extension of~$\Q$ given as the splitting field of the polynomial $X^5+3X^3+6X^2+9$.

Several other examples may also be found in {\em loc. cit.} such as a mod.~$19$ representation of projective image isomorphic to~$S_4$ attached to the unique cusp form of weight~$6$, level~$4$ and trivial Nebentypus. The authors also discuss an effective procedure that given a newform~$f$ and a prime~$\ell$ determines whether some mod.~$\ell$ representation attached to~$f$ has projective image isomorphic to~$A_4$, $S_4$ or $A_5$.


\newcommand{\etalchar}[1]{$^{#1}$}

\end{document}